\documentclass[11pt]{amsart}
\usepackage[a4paper,textwidth=444.53284pt,textheight=630pt,paperheight=845.04684pt,paperwidth=614.295pt,centering]{geometry}

\usepackage[english]{babel}
\usepackage{amsmath,amssymb,amsfonts}
\usepackage{mathrsfs}
\usepackage[all]{xy}
\usepackage[dvipsnames,table]{xcolor}
\usepackage[colorlinks=true,linkcolor=RoyalBlue,citecolor=PineGreen,urlcolor=RoyalBlue]{hyperref}
\usepackage{tikz,pgf}
\usepackage[labelfont=rm]{subcaption}
\usepackage{float}
\usepackage{mathtools}
\usepackage{booktabs}
\usepackage{cite}
\usepackage{MnSymbol}
\usepackage[nameinlink]{cleveref}

\tikzstyle{vertex}=[circle, draw, fill=black, minimum size=5.5pt, inner 
sep=0.8pt]
\newcommand{\vertex}{\node[vertex]}
\tikzstyle{lvertex}=[circle,white!80!black, draw, fill=black, 
font=\small,font=\bfseries, minimum size=10pt, inner sep=0.8pt]
\newcommand{\lvertex}{\node[lvertex]}
\tikzstyle{svertex}=[circle, draw, fill=black, minimum size=3.5pt, inner 
sep=0pt]

\colorlet{col1}{red!70!black}
\colorlet{col2}{green!60!black}
\colorlet{col3}{white!50!black}
\colorlet{col4}{orange!90!black}
\colorlet{col5}{blue!60!black}
\tikzstyle{sedge}=[thick]
\tikzstyle{edge}=[very thick,font=\bfseries]
\tikzstyle{edge1}=[edge,draw=col1]
\tikzstyle{edge2}=[edge,draw=col2]
\tikzstyle{edge3}=[edge,draw=col3]
\tikzstyle{edge4}=[edge,draw=col4]
\tikzstyle{edge5}=[edge,draw=col5]
\tikzstyle{edge_big}=[edge,line width=2.5pt]
\tikzstyle{edge1_big}=[edge_big,draw=col1]
\tikzstyle{edge2_big}=[edge_big,draw=col2]
\tikzstyle{edge3_big}=[edge_big,draw=col3]
\tikzstyle{edge4_big}=[edge_big,draw=col4]
\tikzstyle{edge5_big}=[edge_big,draw=col5]
\tikzstyle{lab1}=[col1]
\tikzstyle{lab2}=[col2]
\tikzstyle{lab3}=[col3]
\tikzstyle{lab4}=[col4]
\tikzstyle{lab5}=[col5]

\newcommand{\colcA}{{\color{col1}\textbf{r}}}
\newcommand{\colcB}{{\color{col2}\textbf{g}}}
\newcommand{\colcC}{{\color{col3}\textbf{d}}}
\newcommand{\colcD}{{\color{col4}\textbf{o}}}
\newcommand{\colcE}{{\color{col5}\textbf{b}}}

\theoremstyle{plain}
\newtheorem{lemma}{Lemma}[section]
\newtheorem{proposition}[lemma]{\textbf{Proposition}}
\newtheorem{theorem}[lemma]{\textbf{Theorem}}
\newtheorem{corollary}[lemma]{\textbf{Corollary}}
\theoremstyle{definition}
\newtheorem{definition}[lemma]{\textbf{Definition}}
\newtheorem{example}[lemma]{\textbf{Example}}
\newtheorem*{notation}{\textbf{Notation}}
\newtheorem*{assumption}{\textbf{Assumption}}
\newtheorem*{terminology}{\textbf{Terminology}}
\newtheorem{remark}[lemma]{Remark}

\newcommand{\N}{\mathbb{N}}
\newcommand{\Z}{\mathbb{Z}}
\newcommand{\Q}{\mathbb{Q}}
\newcommand{\R}{\mathbb{R}}
\newcommand{\C}{\mathbb{C}}
\newcommand{\K}{\mathbb{K}}
\newcommand{\p}{\mathbb{P}}

\newcommand{\edg}{\mcal{E}}
\newcommand{\wt}{\mathrm{wt}}
\newcommand{\Spec}{\operatorname{Spec}}
\newcommand{\Lam}{\operatorname{Lam}}
\newcommand{\leftquot}[2]{{}^{#1}\!{#2}}
\newcommand{\rightquot}[2]{{#2}^{#1}}

\newcommand{\tth}{\thinspace}
\newcommand{\subtract}{\mathop{\backslash}}
\newcommand{\bigsubtract}{\mathop{\big\backslash}}
\newcommand{\quotient}{\mathop{/}}
\newcommand{\bigquotient}{\mathop{\big/}}
\newcommand\bigdotcup[2]{\mathrel{\text{
    \setbox0\hbox{$\bigcup_{#1}^{#2}$}%
    \rlap{\hbox to \wd0{\hss\raisebox{4.5\height}{\large\bfseries .}\hss}}\box0
}}}
\newcommand\dotcup{\mathrel{\text{
    \setbox0\hbox{$\bigcup$}%
    \rlap{\hbox to \wd0{\hss\raisebox{3.5\height}{\large\bfseries .}\hss}}\box0
}}}

\newcommand{\mscr}{\mathscr}
\newcommand{\mcal}{\mathcal}
\renewcommand{\hat}{\widehat}

\title{The number of realizations of a Laman graph}
\author[J. Capco]{%
Jose Capco$^{\dagger, \ast}$}
\author[M. Gallet]{
Matteo Gallet$^{\ast, \circ}$}
\author[G. Grasegger]{
Georg Grasegger}
\author[C. Koutschan]{
Christoph Koutschan$^{\ast, \ddag}$}
\author[N. Lubbes]{
Niels Lubbes$^{\circ}$}
\author[J. Schicho]{
Josef Schicho$^{\ast, \circ}$}

\thanks{The final version of this paper will appear in SIAM Journal on 
Applied Algebra and Geometry (SIAGA). A short summary of this paper 
previously appeared in the conference proceedings~\cite{EurocombVersion}.}

\thanks{$^\dagger$ Supported by the Austrian Science Fund (FWF): P28349}

\thanks{$^\ast$ Supported by the Austrian Science Fund (FWF): W1214-N15, 
Project DK9} 

\thanks{$^\circ$ Supported by the Austrian Science Fund (FWF): P26607}
 
\thanks{$\ddag$ Supported by the Austrian Science Fund (FWF): F5011-N15}

\address[JC, JS]{Research Institute for Symbolic Computation (RISC), Johannes 
Kepler University}

\email{\{jcapco, jschicho\}@risc.jku.at}

\address[MG, GG, CK, NL]{Johann Radon Institute for Computational and Applied 
Mathematics (RICAM), Austrian Academy of Sciences}
\email{\{matteo.gallet, georg.grasseger, 
\newline \phantom{mmmmmmmmmmmmmmmmm} 
christoph.koutschan, niels.lubbes\}@ricam.oeaw.ac.at}

\begin{document}

\begin{abstract}
 Laman graphs model planar frameworks that are rigid for a general choice 
of distances between the vertices. There are finitely many ways, up to 
isometries, to realize a Laman graph in the plane. Such realizations can be 
seen 
as solutions of systems of quadratic equations prescribing the 
distances between pairs of points. Using ideas from algebraic and tropical 
geometry, we provide a recursive formula for the 
number of complex solutions of such systems. 
\end{abstract}
\maketitle

\section*{Introduction}

For a graph~$G$ with edges~$E$, we consider the set of all its realizations
in the plane, such that the lengths of the edges coincide with some
prescribed edge labeling $\lambda\colon E\rightarrow \R_{\geq0}$. 
Edges and vertices are allowed to overlap in such a realization.
For example, suppose that $G$ is the complete graph on four vertices minus one 
edge.
\Cref{figure:realizations} shows all possible realizations of $G$ up to 
rotations and translations, for a particular given edge labeling.
\begin{figure}[H]
\begin{center}
\begin{tikzpicture}[scale=0.6]
 \vertex (a1) at (0,-2) {};
 \vertex (b1) at (0,0) {};
 \vertex (c1) at (2,0) {};
 \vertex (d1) at (2,1) {};
 \begin{scope}[xshift=5cm]
	\vertex (a2) at (0,2) {};
	\vertex (b2) at (0,0) {};
	\vertex (c2) at (2,0) {};
	\vertex (d2) at (2,-1) {};
 \end{scope}
 \begin{scope}[xshift=10cm]
	\vertex (a3) at (0,2) {};
	\vertex (b3) at (0,0) {};
	\vertex (c3) at (2,0) {};
	\vertex (d3) at (2,1) {};
 \end{scope}
 \begin{scope}[xshift=15cm]
	\vertex (a4) at (0,-2) {};
	\vertex (b4) at (0,0) {};
	\vertex (c4) at (2,0) {};
	\vertex (d4) at (2,-1) {};
 \end{scope}
 \path[edge] (a1)edge(b1) (b1)edge(c1) (c1)edge(a1) (b1)edge(d1) (c1)edge(d1);
 \path[edge] (a2)edge(b2) (b2)edge(c2) (c2)edge(a2) (b2)edge(d2) (c2)edge(d2);
 \path[edge] (a3)edge(b3) (b3)edge(c3) (c3)edge(a3) (b3)edge(d3) (c3)edge(d3);
 \path[edge] (a4)edge(b4) (b4)edge(c4) (c4)edge(a4) (b4)edge(d4) (c4)edge(d4);
\end{tikzpicture}
\end{center}
\caption{Realizations of a graph up to rotations and translations.}
\label{figure:realizations}
\end{figure}
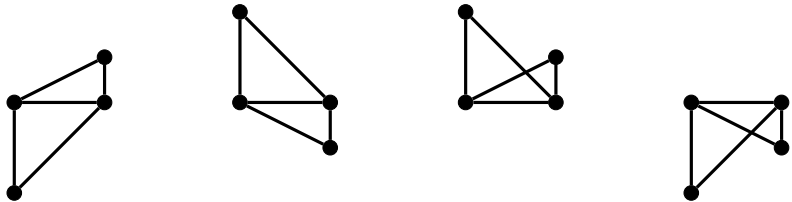
We say that a property holds for a general edge labeling if it holds for all 
edge labelings belonging to the complement of a proper algebraic subset of the 
set of all edge labelings. In this paper we address the following problem:
 \begin{quote}
 {\it For a given graph, determine the number of realizations, up to 
 rotations and translations, for a general edge labeling.}
 \end{quote}

The realizations of a graph can be considered as structures in the plane, which 
are 
comprised of rods connected by rotational joints.
If a graph with an edge labeling admits infinitely (finitely) many realizations 
up to rotations and translations,
then the corresponding planar structure is flexible (rigid), see 
\Cref{figure:starrheit}.

\begin{figure}[H]
  \begin{center}
    \begin{subfigure}[b]{0.3\textwidth}
      \begin{center}
        \begin{tikzpicture}[scale=1.3]
                    \vertex (a) at (0,0) {};
                    \vertex (b) at (1,0) {};
                    \vertex (c) at (0.5,0.866025) {};
                    \vertex (d) at (1.5,0.866025) {};
                    \begin{scope}
                        \vertex[black!40!white] (c2) at (0.7,0.72) {};
                        \vertex[black!40!white] (d2) at (1.7,0.72) {};
                        \draw[edge,black!40!white,dashed] (a) -- (b) -- (d2) -- 
(c2) -- (a) -- cycle;
                    \end{scope}
                    \draw[edge] (a) -- (b) -- (d) -- (c) -- (a) -- cycle;
                \end{tikzpicture}
                \caption{flexible}
                \label{figure:starrheit:flexible}
      \end{center}
    \end{subfigure}
    \begin{subfigure}[b]{0.3\textwidth}
      \begin{center}
                \begin{tikzpicture}[scale=1.3]
                    \vertex (a) at (0,0) {};
                    \vertex (b) at (1,0) {};
                    \vertex (c) at (0.5,0.866025) {};
                    \vertex (d) at (1.5,0.866025) {};
                    \draw[edge] (a) -- (b) -- (d) -- (c) -- (a) -- cycle;
                    \draw[edge] (b) -- (c);
                \end{tikzpicture}
                \caption{rigid}
                \label{figure:starrheit:rigid}
      \end{center}
    \end{subfigure}
    \begin{subfigure}[b]{0.3\textwidth}
      \begin{center}
        \begin{tikzpicture}[scale=1.3]
                    \vertex (a) at (0,0) {};
                    \vertex (b) at (1,0) {};
                    \vertex (c) at (0.5,0.866025) {};
                    \vertex (d) at (1.5,0.866025) {};
                    \draw[edge] (a) -- (b) -- (d) -- (c) -- (a) -- cycle;
                    \draw[edge] (b) -- (c);
                    \draw[edge] (a) -- (d);
                \end{tikzpicture}
                \caption{rigid (overdetermined)}
                \label{figure:starrheit:overdetermined}
      \end{center}
    \end{subfigure}
  \end{center}
  \caption{Graphs and their state of rigidity}
  \label{figure:starrheit}
\end{figure}
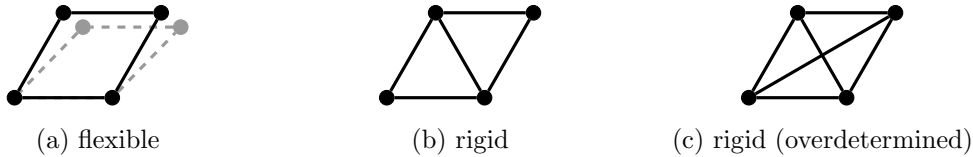

\subsection*{Historical notes}
The study of rigid structures, also called frameworks, was originally
motivated by mechanics and architecture, and goes back as early as the 19th
century to the works of James Clerk Maxwell, August Ritter, Karl Culmann,
Luigi Cremona, August F\"oppl, and Lebrecht Henneberg.
Nowadays, there is still a considerable interest in rigidity
theory~\cite{CombinatRigidity, Connelly1993} due to various applications in 
natural science and engineering; for an exemplary overview, see the conference 
proceedings ``Rigidity Theory and 
Applications''~\cite{RigidityTheoryApplications}.  Let
us just highlight three application areas that are covered there: In materials
science the rigidity of crystals, non-crystalline solids, glasses, silicates,
etc.\  is studied; among the numerous publications in this area we can mention
~\cite{BorceaStreinuCrystal,JacobsHendrickson}.  In
biotechnology one is interested in possible conformations of proteins and
cyclic molecules~\cite{JacobsEtAlProtein}, in particular to the
enumeration of such conformations~\cite{LibertiEtAl,EmirisMourrain}. In
robotics, one aims at computing the configurations of mechanisms, such as 6R
chains or Stewart-Gough platforms.  For the former, the 16 solutions of the
inverse kinematic problem have been found by using very elegant arguments from
algebraic geometry~\cite[Section~11.5.1]{Selig2005}.  For the latter, the 40
complex assembly modes have been determined by algebraic
geometry~\cite{Ronga1995} or by computer algebra~\cite{FaugereLazard1995};
\mbox{Dietmaier}~\cite{Dietmaier1998} showed that there is also an assignment of 
the
parameters such that all $40$~solutions are real. 
Recently, connections between rigidity theory and incidence problems have been
established~\cite{Raz2016}.

\subsection*{Pollaczek-Geiringer's and Laman's characterization}

A graph is called generically rigid (or isostatic) if a general edge labeling 
yields a rigid realization. No edge in a generically rigid graph can be 
removed without losing rigidity. This is why such graphs are also called 
minimally rigid in the literature. Note that the graph in 
\Cref{figure:starrheit:overdetermined} is not generically rigid, while the 
one in  \Cref{figure:starrheit:rigid} is. Hilda 
Pollaczek-Geiringer~\cite{Geiringer1927} characterized this property in terms 
of 
the number of edges and vertices of the graph and its subgraphs. The same 
characterization can be found in a paper of Gerard Laman~\cite{Laman1970} more 
than 40 years later. 
Unfortunately, the results of Pollaczek-Geiringer have been unnoticed until 
recently.
Nowadays, these objects are known as Laman graphs; since this terminology is 
well-known, we stick to it in this paper.

\subsection*{State of the art}
All realizations of a Laman graph with an edge labeling
can be recovered as the solution set of a system of algebraic equations, 
where the edge labels can be seen as parameters. 
Here, we are interested in the number of 
complex solutions of such a system, up to an equivalence relation coming from 
direct planar isometries; this number is the same for any general choice of 
parameters, 
so we call it the {\em Laman number} of the graph. For some graphs up to~$8$ 
vertices, this number has been computed using random values for the 
parameters~\cite{JacksonOwen2012} --- this means that it is very likely, but 
not 
absolutely certain, that these computations give the true numbers. Upper and 
lower bounds on the maximal Laman number for graphs with up to~$10$ vertices 
were found by analyzing the Newton polytopes of the equations and their mixed 
volumes~\cite{Emiris2009} using techniques from \cite{Steffens2010}. It has 
been 
proven~\cite{Borcea2004} that the Laman number of a Laman graph with~$n$ 
vertices is at most~$\binom{2n-4}{n-2}$.

\subsection*{Our contribution}
Our main result is a combinatorial algorithm that 
computes the number of complex realizations of any given Laman graph.
This is much more efficient than just solving the 
corresponding nonlinear system of equations.

We found it convenient to see systems of equations related to Laman graphs as 
special cases of a slightly more general type of systems, determined by 
\emph{bigraphs}. Roughly stated, a bigraph is a pair of graphs whose edges
are in bijection. Every graph can be turned into a bigraph by 
duplication and it is possible to extend the notion of Laman number also 
to bigraphs. The majority of these newly introduced systems do not have 
geometric significance: they are merely introduced to have a 
suitable structure to set up a recursive strategy. Our main result 
(\Cref{theorem:laman_number}) is a recursive formula 
expressing the Laman number of a bigraph in terms of Laman numbers of smaller 
bigraphs. Using this formula we succeeded in computing the exact Laman numbers
of graphs with up to~$18$ vertices --- a task that was absolutely out of reach
with the previously known methods.

The idea for proving the recursive formula is inspired by tropical geometry 
(see~\cite{Maclagan2015} or \cite[Chapter~9]{Sturmfels2002}): we consider the 
system of equations over the field of Puiseux series, and the inspection of the 
valuations of the possible solutions allows us to endow every bigraph with some 
combinatorial data that prescribes how the recursion should proceed. This
gives, therefore, a recursive formula for the right hand side of 
Corollary~3.6.16 
in~\cite{Maclagan2015} in our particular case.
Notice that the Laman number of a graph can be understood as the base degree 
(as defined at the end of Section~$1$ of~\cite{Rosen2014}) of the algebraic 
matroid 
associated to the variety parametrized by the square distances of the pairs
of points prescribed by the edges of the Laman graph.

\subsection*{Structure of the paper}
\Cref{laman_graphs} contains the statement 
of the problem and a proof of the equivalence of generic rigidity and Laman's 
condition in our setting. This section is meant for a general mathematical 
audience and requires almost no prerequisite. \Cref{bigraphs} analyzes 
the system of equations defined by a bigraph, and \Cref{bidistances} 
provides a general formulation for a recursive formula for the number of 
solutions of the 
system. Here, we employ some standard techniques in algebraic geometry, so the 
reader should be acquainted with the basic concepts in this area. In 
\Cref{formulas}, we specialize the general result provided at the end 
of \Cref{bidistances} and we give a recursive formula for the Laman 
number. It leads to an algorithm that is employed in \Cref{computations}
to derive some new results on the number of realizations of Laman graphs. These 
last two sections are 
again meant for a general audience, and they require only the knowledge of the 
objects and the results in \Cref{bigraphs,bidistances}, but 
not of the proof techniques used there.
For a condensed and streamlined version of this paper, we refer to the extended 
abstract~\cite{EurocombVersion}.

\subsection*{Acknowledgments}
We thank Jan Peter Sch\"afermeyer for making us aware of the work by Hilda
Pollaczek-Geiringer. We are grateful for the careful proofreading of the two
anonymous referees, and for their encouraging and constructive suggestions.
We also thank Bill Jackson for his careful reading resulting in an improvement
of some technical aspects.

\section{Laman graphs}
\label{laman_graphs}

In this section, by a \emph{graph} we mean a finite, connected, undirected 
graph without self-loops or multiple edges. We write $G = (V,E)$ to denote 
a graph~$G$ with set of vertices~$V$ and set of edges~$E$. An (unoriented) 
edge~$e$ 
between vertices~$u$ and~$v$ is denoted by~$\{u,v\}$.

\begin{definition}
\label{definition:labeling}
A \emph{labeling} of a graph~$G = (V,E)$ is a function 
$\lambda \colon E \longrightarrow \R$; the pair $(G, \lambda)$ 
is called a \emph{labeled graph}.
A \emph{realization} of~$G$ is a function $\rho\colon V \longrightarrow\R^2$.
We say that a realization~$\rho$ is \emph{compatible with}~a labeling~$\lambda$ 
if
for each edge $e \in E$ the Euclidean distance between its endpoints agrees with 
its label:
\begin{equation}\label{equation:compatibility}
 \lambda(e) \, = \, \bigl\| \rho(u)-\rho(v) \bigr\|^2, \quad \text{where } e = 
\{u,v\}.
\end{equation}
A labeled graph $(G, \lambda)$ is \emph{realizable} if and only if there 
is a realization compatible with~$\lambda$.
\end{definition}

\begin{definition}
\label{definition:equivalent_realization}
We say that two realizations~$\rho_1$ and $\rho_2$ of a graph~$G$ are 
\emph{equivalent} if and only if there exists a direct Euclidean 
isometry~$\sigma$ of~$\R^2$ such that~$\rho_1 =\sigma\circ\rho_2$;
a direct Euclidean isometry is an affine-linear
map $\R^2\longrightarrow\R^2$ that preserves distance and orientation in~$\R^2$.
\end{definition}

\begin{definition}
\label{definition:rigid_graph}
 A labeled graph $(G, \lambda)$ is called \emph{rigid} if it satisfies the
following properties:
\begin{itemize}
 \item
 $(G, \lambda)$ is realizable;
 \item
 there are only finitely many realizations compatible with~$\lambda$, up to 
equivalence.
\end{itemize}
\end{definition}


Our main interest is to count the number of realizations of generically rigid
graphs, namely graphs for which almost all realizable labelings induce rigidity.
Unfortunately, in the real setting, this number is not well-defined, since it
may depend on the actual labeling and not only on the graph.
In order to define a number that depends only on the graph, we switch to the
complex setting.  By this we mean that we allow complex labelings
$\lambda\colon E\longrightarrow\C$ and complex realizations $\rho\colon
V\longrightarrow\C^2$. In this case, the compatibility
condition~\Cref{equation:compatibility} becomes
\[
 \lambda(e) \, = \, \bigl\langle \rho(u)-\rho(v), \rho(u)-\rho(v) 
\bigr\rangle,\quad  e = \{u,v\},
\]
where $\left\langle x, y \right\rangle = x_1y_1 + x_2y_2$.  Moreover, we
consider ``direct complex isometries'', namely maps
\[
  \Bigl(\hspace{-4pt}\begin{array}{c} x \\[-4pt] y\end{array}\hspace{-4pt}\Bigr) 
\longmapsto
  A\,
  \Bigl(\hspace{-4pt}\begin{array}{c} x \\[-4pt] y\end{array}\hspace{-4pt}\Bigr)
  + b,
  \quad A\in\C^{2\times 2} \text{ and } b\in\C^2,
\]
where $A$ is an orthogonal matrix with determinant~$1$. Here, the word 
``isometries'' is an abuse of language, since in this case $\left\langle \cdot, 
\cdot \right\rangle$ is not an inner product. Notice that if we are
given a labeling $\lambda \colon E \longrightarrow \R$ for a graph~$G$ and two
realizations of~$G$ into~$\R^2$ that are not equivalent under real direct
isometries, then they are also not equivalent under complex
isometries. This means that counting the number of non-equivalent
realizations in~$\C^2$ delivers an upper bound for the number of non-equivalent
realizations in~$\R^2$.

\begin{terminology}
Given a graph $G = (V,E)$, the set of possible labelings $\lambda \colon E 
\longrightarrow \C$ forms a vector space, that we denote by~$\C^{E}$. In this 
way 
we are able to address the components of a vector~$\lambda$ in~$\C^{E}$ 
by edges~$e \in E$, namely by writing $\lambda = (\lambda_e)_{e \in E}$. Since 
$\C^{E}$ is a vector space, it is meaningful to speak about properties holding 
for a \emph{general} labeling: a property~$\mscr{P}$ holds for a general 
labeling if the set
\[
 \bigl\{ \lambda \in \C^{E} \, : \, \mscr{P}(\lambda) \text{ does not hold} 
\bigr\}
\]
is contained in a proper algebraic subset of~$\C^E$, i.e.\ a subset 
strictly contained in~$\C^E$ and defined by polynomial equations.
\end{terminology}

\begin{definition}
\label{definition:generically_rigid}
 A graph~$G$ is called \emph{generically realizable} if for a general 
labeling~$\lambda$ the labeled graph $(G, \lambda)$ is realizable. A graph~$G$ 
is called \emph{generically rigid} if for a general labeling~$\lambda$ the 
labeled graph $(G, \lambda)$ is rigid.
\end{definition}

\begin{remark}
\label{remark:generically_realizable}
 If a graph $G$ is generically realizable, then every subgraph~$G'$ of~$G$ is 
generically realizable. Every general labeling for~$G'$ can 
be extended to a general labeling for~$G$. Since by hypothesis 
$G$ has a compatible realization, the subgraph~$G'$ admits such a
realization as well.
\end{remark}

\begin{definition}
 \label{definition:laman_graph}
 A \emph{Laman graph} is a graph $G=(V,E)$ such that $\vert E\vert = 2\vert 
V\vert-3$, and for every subgraph $G'=(V',E')$ it holds $\vert E' \vert \leq
2\vert V' \vert-3$.
\end{definition}

We are going to see (\Cref{theorem:laman}) that Laman graphs are
exactly the generically rigid ones. Many different characterizations of this
property have appeared in the literature, for example by construction 
steps~\cite{Henneberg} (see \Cref{theorem:laman}), or in terms of 
spanning trees after doubling one edge~\cite{LovaszYemini} or after adding an
edge~\cite{Recski}, or in terms of three trees such that each vertex of the
graph is covered by two trees~\cite{Crapo2006}.
These characterizations can be used for decision algorithms on the minimal 
rigidity of a given 
graph~\cite{Bereg,JacobsHendrickson,DaescuKurdia,Gortler2010}.

For any graph $G=(V,E)$, there is a natural map~$r_{G}$ from the 
set~$\C^{2|V|}$ 
of its realizations to the set~$\C^{E}$ of its labelings:
\[
  r_G\colon \C^{2|V|} \longrightarrow \C^{E}, \quad
  (x_v, y_v)_{v \in V} \; \longmapsto \; \bigl( (x_u - x_v)^2 + (y_u - y_v)^2 
\bigr)_{\{u,v\} \in E}\,.
\]
Each fiber of~$r_G$, i.e.\ a preimage $r_G^{-1}(p)$ of a single 
point~$p\in\C^E$,
is invariant under the group of direct complex isometries. We define a 
subspace $\C^{2|V|-3} \subseteq \C^{2|V|}$ as follows: choose two distinguished
vertices~$\bar{u}$ and~$\bar{v}$ with $\{\bar{u},\bar{v}\}\in E$,
and consider the linear subspace defined by the equations $x_{\bar{u}} 
= y_{\bar{u}} = 0$ and ${x_{\bar{v}} = 0}$. In this way, the 
subspace~$\C^{2|V|-3}$ 
intersects every orbit of the action of isometries on a fiber 
of~$r_G$ in exactly two points: in fact, the equations do not allow
any further translation or rotation; however, for any labeling $\lambda \colon 
E \longrightarrow \C$ and for every realization in~$\C^{2|V|-3}$ compatible 
with~$\lambda$ there exists another realization, obtained by multiplying the 
first 
one by~$-1$, which is equivalent, but gives a different point 
in~$\C^{2|V|-3}$. The restriction of~$r_G$ to~$\C^{2|V|-3}$ 
gives the map 
\[ 
 h_G \colon \C^{2|V|-3} \longrightarrow \C^{E}. 
\]
The following statement follows from the construction of~$h_G$; notice that
the choice of~$\bar{u}$ and~$\bar{v}$ has no influence on the result.  Recall
that a map $f\colon X\longrightarrow Y$ between algebraic sets is called
\emph{dominant} if $Y\setminus f(X)$ is contained in an algebraic proper subset 
of~$Y$.

\begin{lemma}
\label{lemma:map_h}
 A graph~$G$ is generically rigid if and only if $h_G$ is dominant and a 
general fiber of~$h_G$ is finite. This is equivalent to saying that $h_G$ is 
dominant and $2|V| = |E|+3$.
\end{lemma}
\begin{proof}
 It is enough to notice that if $h_G$ is dominant, then the dimension of the 
general fiber is $2|V| - 3 - |E|$.
\end{proof}

We state Laman's theorem characterizing generically rigid graphs. A proof, 
which 
closely follows Laman's original argument in his paper~\cite{Laman1970}, can be 
found in \Cref{section:prooflaman}. For our 
purposes, we need a result that implies the existence of only a finite number 
of 
complex realizations, while the original statement deals with the real setting 
and proves that a given realization does not admit infinitesimal deformations.

\begin{figure}
\begin{center}
\begin{subfigure}[t]{0.45\textwidth}
\begin{center}
\begin{tikzpicture}[scale=0.65]
 \draw [dashed] (0,0) ellipse (2cm and 1cm);
 \vertex at (-1,0) [label=below:$u$] {};
 \vertex at (1,0) [label=below:$v$] {};
 \draw [dashed] (6,0) ellipse (2cm and 1cm);
 \vertex (t) at (6,2) [label=above:$t$] {};
 \vertex (u) at (5,0) [label=below:$u$] {};
 \vertex (v) at (7,0) [label=below:$v$] {};
 \path
  (u) edge[edge] (t)
  (t) edge[edge] (v)
 ;
\end{tikzpicture}
\end{center}
\caption{The first Henneberg rule: given any two vertices~$u$ and~$v$
(which may be connected by an edge or not), we add a 
vertex~$t$ and the two edges $\{u,t\}$ and $\{v,t\}$.}
\label{figure:henneberg_1}
\end{subfigure}
\qquad
\begin{subfigure}[t]{0.45\textwidth}
\begin{center}
\begin{tikzpicture}[scale=0.65]
 \draw [dashed] (0.7,0) ellipse (1.8cm and 1.8cm);
 \vertex at (0.7,1) [label=above:$w$] {};
 \vertex (u) at (-0.3,-0.5) [label=below:$u$] {};
 \vertex (v) at (1.7,-0.5) [label=below:$v$] {};
 \draw [dashed] (6,0) ellipse (1.8cm and 1.8cm);
 \vertex (uu) at (5,-0.5) [label=below:$u$] {};
 \vertex (vv) at (7,-0.5) [label=below:$v$] {}; 
 \vertex (ww) at (6,1) [label=above:$w$] {};
 \vertex (tt) at (9,1) [label=above:$t$] {};
 \path
   (u) edge[edge] (v)
   (uu) edge[edge] (tt)
   (vv) edge[edge] (tt)
   (ww) edge[edge] (tt)
 ;
\end{tikzpicture}
\end{center}
\caption{The second Henneberg rule: given any three vertices $u$, $v$, and~$w$ 
such that $u$ and $v$ are connected by an edge, we remove the edge $\{u,v\}$, 
we add a vertex~$t$ and the three edges $\{u,t\}$, $\{v,t\}$, and 
$\{w,t\}$.}
\label{figure:henneberg_2}
\end{subfigure}
\end{center}
\caption{Henneberg rules}
\end{figure}

\begin{theorem}
\label{theorem:laman}
Let $G$ be a graph. Then the following three conditions are equivalent:
\begin{enumerate}
\item\label{it:laman} $G$ is a Laman graph;
\item\label{it:generically} $G$ is generically rigid;
\item\label{it:henneberg}
$G$ can be constructed by iterating the two \emph{Henneberg rules} (see 
\Cref{figure:henneberg_1,figure:henneberg_2}), starting from the
graph that consists of two vertices connected by an edge.
\end{enumerate}
\end{theorem}

Given a Laman graph, we are interested in the number of its realizations 
in~$\C^2$ that are compatible with a general labeling, up to equivalence. As we 
have already pointed out, the degree of the map~$h_G$ (namely, the cardinality 
of a fiber~$h_G^{-1}(p)$ over a general point~$p$) is twice the number of 
realizations of~$G$ compatible with a general labeling, up to 
equivalence. \Cref{lemma:map_h} and \Cref{theorem:laman} imply that the 
map~$h_G$ is dominant and its degree is finite. 

We now construct a map whose degree is exactly the number of equivalence 
classes. For this purpose, we employ a different way than in~$h_G$ to get rid 
of 
complex ``translations'' and ``rotations'': first, for the translations, we 
take 
a quotient of vector spaces, which can be interpreted as setting 
$x_{\bar{u}}=y_{\bar{u}}=0$ as for~$h_G$, or alternatively as moving the 
barycenter of a realization to the origin; second, we use projective 
coordinates 
to address the rotations. 
More precisely, in order to study the system of equations
\[
  (x_u - x_v)^2 + (y_u - y_v)^2 \, = \, \lambda_{uv} 
  \quad \text{for all } \{u,v\} \in E,
\]
which defines a realization of a Laman graph, we can regard the vectors 
$(x_u)_{u \in V}$ and $(y_u)_{u \in V}$ as elements of the space $\C^{V} 
\quotient \left\langle x_u = x_v \text{ for all } u,v \in V \right\rangle$. In 
this way, we are allowed to add arbitrary constants to all components~$x_u$ or 
to all components~$y_u$ without changing the representative in the quotient;
hence these vectors are invariant under translations. 
Moreover, if one performs the change of variables
\begin{equation}\label{equation:eta}
 (x_v)_{v \in V}, (y_v)_{v \in V} \quad \longmapsto \quad (x'_v := x_v + i 
\tth y_v)_{v \in V},\, (y'_v :=  x_v - i \tth y_v)_{v \in V},
\end{equation}
then the previous system of equations becomes
\[
  (x'_u - x'_v)(y'_u - y'_v) \, = \, \lambda_{uv} 
  \quad \text{for all } \{u,v\} \in E
\]
and the action of a complex rotation turns into the multiplication of the 
$x'_u$\hbox{-}coordi\-nates by a scalar in~$\C$, and of the 
$y'_u$\hbox{-}coordinates by its 
inverse. Thus, by considering $(x'_u)_{u \in V}$ and $(y'_u)_{u \in V}$ as 
coordinates in two different projective spaces, the points we obtain are 
invariant under complex rotations.
In order to employ these two strategies, we define
\[
 \p_\C^{|V|-2} := \p \Bigl( \C^{V} \big/ \; \bigl\langle(1,\dots,1)\bigr\rangle 
\Bigr) =
  \p \Bigl( \C^{V} \big/ \; \bigl\langle (x_v)_{v \in V} \, : \,  
  x_u = x_w \text{ for all } u,w \in V \bigr\rangle \Bigr)
\]
and the map 
\begin{equation}
\label{equation:f_G}
 \begin{array}{rrcl}
 f_G \colon & \p_\C^{|V|-2} \times \p_\C^{|V|-2} & \dashrightarrow & 
\p_\C^{|E|-1} \\
 & [(x_v)_{v \in V}], [(y_v)_{v \in V}] & \longmapsto & \Bigl( (x_u - x_v)(y_u 
- 
y_v) \Bigr)_{\{u,v\} \in E}
 \end{array},
\end{equation}
where $[\, \cdot \,]$ denotes the point in~$\p_\C^{|V|-2}$ determined by a 
vector
in~$\C^V$. Notice that the map~$f_G$ is well-defined, because the
quantities~$x_u - x_v$ depend, up to scalars, only on the
points~$[(x_v)_{v \in V}]$, and not on the particular choice of
representatives (and similarly for~$y_u - y_v$).  Note that $f_G$ may not
be defined everywhere, which is conveyed by the notation~$\dashrightarrow$.

\begin{lemma} 
\label{lemma:fiber}
For any Laman graph~$G$ the equality ${\deg}{\left( h_G \right)} = 2 \tth 
{\deg}{\left( f_G \right)}$ holds.
\end{lemma}

\begin{proof}
Recall that the degree is computed by counting the number of preimages of a 
general point in the codomain. Let therefore $\lambda \in \C^E$ be a general 
labeling and let $\{\bar{u}, \bar{v}\}$ be the edge used to define~$h_G$, so in 
particular we can suppose $\lambda_{\{\bar{u},\bar{v}\}} \neq 0$.
We show that there is a 2:1 map~$\eta$ from~$h_G^{-1}(\lambda)$ 
to~$f_G^{-1}(\lambda)$, where $\lambda \in \p_\C^{|E|-1}$ is the point defined 
by the values of~$\lambda \in \C^E$ as projective coordinates. The map~$\eta$
is defined according to the change of variables~\Cref{equation:eta}:
\[
  \eta \colon 
  h_G^{-1}(\lambda) \longrightarrow f_G^{-1}(\lambda), \quad
  (x_v)_{v \in V}, (y_v)_{v \in V} \; \longmapsto \;
  \bigl[(x_v + i \tth y_v)_{v \in V}\bigr],\bigl[(x_v - i \tth y_v)_{v \in 
V}\bigr].
\]
In other words, we just take the 
coordinates of the embedded vertices as projective coordinates and make a 
complex coordinate transformation, namely one that diagonalizes the linear part 
of the  isometries. The map~$\eta$ is well-defined, since the 
quantities~$(x_v + i \tth y_v)_{v \in V}$ and~$(x_v - i \tth y_v)_{v 
\in V}$ are never all zero because of the definition of the map~$h_G$.
For $q \in \p_\C^{|V|-2} \times \p_\C^{|V|-2}$ of the form $q = \bigl( 
[(\hat{x}_v)_{v \in V}], [(\hat{y}_v)_{v \in V}] \bigr)$ and such that 
$\hat{x}_{\bar{u}} \ne \hat{x}_{\bar{v}}$ and 
$\hat{y}_{\bar{u}} \ne \hat{y}_{\bar{v}}$, we choose coordinates 
$(\hat{x}_v)_{v \in V}, (\hat{y}_v)_{v \in V}$ such that 
$\hat{x}_{\bar{u}}=\hat{y}_{\bar{u}}=0$, $\hat{x}_{\bar{v}}=1$, and
$\hat{y}_{\bar{v}}=-1$. This is possible because we can add a constant vector 
to any of~$(\hat{x}_v)_{v \in V}$ or~$(\hat{y}_v)_{v \in V}$
without changing the point in $\p_\C^{|V|-2}\times\p_\C^{|V|-2}$. When $q 
\in f_G^{-1}(\lambda)$, every point in~$\eta^{-1}(q)$ is of the form 
$\bigl( (x_v)_{v \in V}, (y_v)_{v \in V} \bigr)$, where $x_{\bar{u}} = 
y_{\bar{u}} = 0$ and $x_{\bar{v}} = 0$ (recall the definition of the 
map~$h_G$). By definition of~$\eta$, we have that for all $v \in V$:
\[
 \left\{
 \begin{array}{rcl}
  x_v + i \tth y_v &=& c \tth \hat{x}_v, \\
  x_v - i \tth y_v &=& d \tth \hat{y}_v,
 \end{array}
 \right.
 \qquad \text{hence} \qquad
 \left\{
 \begin{array}{rcl}
  x_v &=& (c \tth \hat{x}_v + d \tth \hat{y}_v) / 2, \\
  y_v &=& (c \tth \hat{x}_v - d \tth \hat{y}_v) / 2i,
 \end{array}
 \right.
\]
for some constants $c,d \in \C$. Thus, for $v = \bar{v}$, we get the equation 
$0 = c-d$, which in turn implies that every point in~$\eta^{-1}(q)$ determines 
a 
realization of the form
\[ 
 \rho \colon V \longrightarrow \C^2, \qquad 
 v \; \longmapsto \; \left( c \tth \frac{\hat{x}_v + \hat{y}_v}{2}, c \tth 
\frac{\hat{x}_v - \hat{y}_v}{2i} \right),
\]
that must be compatible with~$\lambda$. By construction, the 
constant~$c$ must satisfy $c^2 = -\lambda_{\{\bar{u},\bar{v}\}}$, 
since $\lambda_{\{\bar{u},\bar{v}\}} = \left\langle \rho(\bar{u}) 
- \rho(\bar{v}),\rho(\bar{u}) - \rho(\bar{v}) \right\rangle$. 
There are exactly two such numbers~$c$, and this proves the statement.
\end{proof}

\begin{corollary}
The number of realizations of a Laman graph, compatible with a
general labeling and counted up to equivalence, is equal to the degree of the 
map~$f_G$.
\end{corollary}

\section{Bigraphs and their equations}
\label{bigraphs}

In this section we introduce the main concept of the paper, the one of 
\emph{bigraph}. Bigraphs are pairs of graphs whose edges are in bijection. 
Every graph determines a bigraph by simply duplicating 
it and considering the natural bijection between the edges. It is possible to 
associate to any bigraph a rational map as we did with the map~$f_G$ 
in~\Cref{equation:f_G}.
The reason for this duplication is that, in order 
to set up a recursive formula for the degree of~$f_G$, we want to be able to 
handle independently the two factors~$(x_u - x_v)$ and~$(y_u - y_v)$ that 
appear in its specification. To do this, we have to allow disconnected graphs 
with multiple edges. 

Notice that if we allow graphs with multiedges, then we have to give away the 
possibility to encode an edge via an unordered pair of vertices. Instead, we 
consider the sets~$V$ and $E$ of vertices and edges, respectively, to be 
arbitrary sets, related by a function $\tau \colon E \longrightarrow 
\mcal{P}(V)$, where $\mcal{P}$ denotes the power set, assigning to each edge 
its corresponding vertices. The image of an element $e \in E$ via~$\tau$ can be 
either a set of cardinality two, when $e$ connects two distinct vertices, or a 
singleton, when $e$ is a self-loop. This way of encoding graphs allows to use 
the same set for the edges of two graphs; this realizes formally the idea of 
prescribing a bijection between the edges of two graphs.

\begin{definition}
\label{definition:bigraph}
 A \emph{bigraph} is a pair of finite undirected graphs $(G, H)$ --- allowing 
several components, multiple edges and self-loops --- where $G = (V, \edg)$ and 
$H = (W, \edg)$. We denote by $\tau_G \colon \edg \longrightarrow \mcal{P}(V)$
and $\tau_H \colon \edg \longrightarrow \mcal{P}(W)$ the two maps assigning to 
each edge its vertices. The set~$\edg$ is called the set of \emph{biedges}.
For technical reasons, we need to order the vertices of edges in~$G$ or~$H$; 
therefore, we assume that there is a total order~$\prec$ given on the sets of 
vertices~$V$ and~$W$. An example of a bigraph is provided in 
\Cref{figure:4_laman}.
\end{definition}

Notice that a single graph $G = (V,E)$ can be turned into
a bigraph by considering the pair $(G,G)$, and by taking the set of 
biedges to be~$E$; the total order~$\prec$ is obtained by fixing any total 
order on~$V$ and duplicating it. Next, we extend a weakened version of the Laman
condition to bigraphs.

\begin{definition}
\label{definition:dim_graph}
For a graph $G=(V,E)$ we define the \emph{dimension} of~$G$ as
\[
 \dim(G) \, := \, |V| - |\{ \text{connected components of } G \}|.
\]
\end{definition}

\begin{remark}
\label{remark:laman_dimension}
 Since a Laman graph is connected by assumption, the condition $2|V| = |E| + 3$ 
can be rewritten as $2 \dim(G) = |E| + 1$.
\end{remark}

\begin{definition}
\label{definition:pseudo_laman}
Let $B = (G, H)$ be a bigraph with biedges~$\edg$, then we say that $B$ is 
\emph{pseudo-Laman} if
\[
 \dim(G) + \dim(H) \, = \, |\edg| + 1.
\]
\end{definition}
It follows from \Cref{remark:laman_dimension} that for 
a Laman graph $G$ the bigraph~$(G,G)$ is pseudo-Laman.

We introduce two operations that can be performed on a graph, starting from a 
subset of its edges: the subtraction of edges and the quotient by edges. We are 
going to use these constructions several times in our paper: subtraction is 
first used at the end of this section, while the quotient operation is mainly 
utilized starting from \Cref{bidistances}.

\begin{definition}
\label{definition:quotient_graphs}
 Let $G = (V, E) $ be a graph, and let $E' \subseteq E$. We define two 
new graphs, denoted $G \quotient E'$ and $G \subtract E'$, as follows.
An example for the operations is provided in 
\Cref{figure:quotient_graphs}.
\begin{itemize}
 \item
  Let $G'$ be the subgraph of $G$ determined by~$E'$. We define $G 
\quotient E'$ to be the graph obtained as follows. Its vertices are the 
equivalence classes of the vertices of~$G$ modulo the relation dictating that 
two vertices~$u$ and~$v$ are equivalent if there exists a path in~$G'$ 
connecting them. Its edges are determined by edges in~$E \setminus E'$, more 
precisely an edge~$e$ in~$E \setminus E'$ such that $\tau_{G}(e) = 
\{u,v\}$ defines an edge in the quotient connecting the equivalence classes 
of~$u$ and~$v$ if and only $e$ is not an edge of~$G'$.
 \item
  Let $\hat{V}$ be the set of vertices of~$G$ that are endpoints of some
edge in~$E \setminus E'$. Define $G \subtract E' = (\hat{V}, E \setminus E')$.
\end{itemize}
\end{definition}

\begin{figure}
 \begin{center}
 \begin{subfigure}[t]{0.33\textwidth}
 \centering
 \begin{tikzpicture}[scale=1.5]
  \vertex (a) at (0,0) {};
  \vertex (b) at (1,0) {};
  \vertex (c) at (1,1) {};
  \vertex (d) at (0,1) {};
  \vertex (e) at (2,0) {};
  \vertex (f) at (2,1) {};
  
  \path[]
   (d) edge[edge1, dashed] (a)
   (a) edge[edge] (b)
   (b) edge[edge] (c)
   (e) edge[edge] (f)
  ;
 \draw[edge1, dashed] (d) .. controls (0.5, 0.8) .. (c);
 \draw[edge] (d) .. controls (0.5, 1.2) .. (c);
 \end{tikzpicture}
 \caption{A graph $G = (V,E)$ and a subset~$E'$ of edges, in dashed red.}
 \end{subfigure}
 \begin{subfigure}[t]{0.3\textwidth}
    \centering
    \begin{tikzpicture}[scale=1.5]
     \vertex (a) at (0,0) {};
     \vertex (b) at (0,1) {};
     \vertex (e) at (1,0) {};
     \vertex (f) at (1,1) {};
  
     \path[]
      (e) edge[edge] (f)
     ;
     \draw[edge] (a) .. controls (-0.2, 0.5) .. (b);
     \draw[edge] (a) .. controls (0.2, 0.5) .. (b);
     \draw[edge] (b) .. controls (0.4, 1.4) and (-0.4, 1.4) .. (b);
    \end{tikzpicture}
    \caption{The graph $G \quotient E'$.}
  \end{subfigure}
  \begin{subfigure}[t]{0.3\textwidth}
    \centering
    \begin{tikzpicture}[scale=1.5]
     \vertex (a) at (0,0) {};
     \vertex (b) at (1,0) {};
     \vertex (c) at (1,1) {};
     \vertex (d) at (0,1) {};
     \vertex (e) at (2,0) {};
     \vertex (f) at (2,1) {};
  
     \path[]
      (d) edge[edge] (c)
      (a) edge[edge] (b)
      (b) edge[edge] (c)
      (e) edge[edge] (f)
     ;
    \end{tikzpicture}
    \caption{The graph $G \subtract E'$.}
  \end{subfigure}
 \end{center}
\caption{Example of the two constructions in 
\Cref{definition:quotient_graphs}.}
\label{figure:quotient_graphs}
\end{figure}
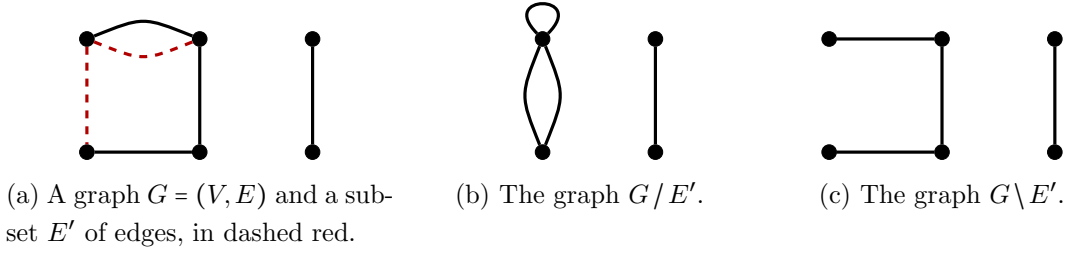

Via \Cref{definition:bigraph_space,definition:bigraph_map} 
we associate to each bigraph~$B$ a rational map~$f_B$, as we did in 
\Cref{laman_graphs} for graphs. 

\begin{definition}
\label{definition:bigraph_space}
Let $B = (G, H)$ be a bigraph, where $G = (V, \edg)$ and $H = (W, \edg)$. 
We set
\begin{align*}
 \p_\C^{\dim(G) - 1} & := \p \bigl( \C^V \bigquotient L_G \bigr), &
 \p_\C^{\dim(H) - 1} & := \p \bigl( \C^W \bigquotient L_H \bigr),
\end{align*}
where
\begin{align*}
 L_G & := \left\langle (x_v)_{v \in V} \, : \,
  \begin{array}{c}
    x_u = x_t \text{ if and only if } u \text{ and } t \\
   \text{are in the same connected component of } G
  \end{array}
 \right\rangle, \\
 L_H & := \left\langle (y_w)_{w \in W} \, : \,
  \begin{array}{c}
    y_u = y_t \text{ if and only if } u \text{ and } t \\
   \text{are in the same connected component of } H
  \end{array}
 \right\rangle,
\end{align*}
and $(x_v)_{v \in V}$ are the standard coordinates of~$\C^{V}$ and similarly 
for $(y_w)_{w \in W}$.
\end{definition}

\begin{definition}
\label{definition:bigraph_map}
 Let $B = (G, H)$ be a bigraph, where $G = (V, \edg)$ and $H = (W, \edg)$. 
Define 
\[
 \begin{array}{rrcl}
  f_{B} \colon & \p_\C^{\dim(G)-1} \times \p_\C^{\dim(H)-1} & \dashrightarrow & 
\p_\C^{|\edg|-1} \\
  & [(x_v)_{v \in V}], [(y_w)_{w \in W}] & \longmapsto & \bigl( (x_u - x_v)(y_t 
- y_w) \bigr)_{e \in \edg}
 \end{array},
\]
where $\{u,v\}=\tau_G(e)$, $u\prec v$, and $\{t,w\}=\tau_H(e)$, $t\prec w$, 
with $\tau_G$ and $\tau_H$ as in \Cref{definition:bigraph}. Here and 
in the rest of the paper, if $e$ is a self-loop say in~$G$, then the 
corresponding polynomial in the definition of~$f_B$ is considered to be $x_u - 
x_u = 0$.
As in \Cref{laman_graphs}, the square brackets~$[ \, \cdot \,]$ denote 
points in~$\p_\C^{\dim(G)-1}$ or~$\p_\C^{\dim(H)-1}$ determined by vectors 
in~$\C^V$ 
or~$\C^W$. As for the map~$f_G$, the map~$f_B$ is well-defined 
because the quantities $(x_u - x_v)$ and $(y_t - y_w)$ depend only, up to 
scalars, on points in~$\p_\C^{\dim(G)-1}$ and~$\p_\C^{\dim(H)-1}$, 
and not on the chosen representatives. We call the map~$f_B$ the \emph{rational 
map associated to~$B$}.
\end{definition}

In \Cref{definition:bigraph_map} we impose $u \prec v$ and $t \prec
w$ in the equations defining the map~$f_B$. The reason for this is that we
want $f_B$, when $B$ is of the form~$(G,G)$, to coincide with~$f_G$ defined
at the end of \Cref{laman_graphs}. If we do not specify the order in
which the vertices appear in the expressions $(x_u - x_v)$ and $(y_t - y_w)$,
we could end up with a map~$f_B$ for which one component is of the form $(x_u
- x_v)(y_v - y_u)$, and not $(x_u - x_v)(y_u - y_v)$ as we would expect.
As in \Cref{laman_graphs}, we are mainly interested in the degree 
of the rational map associated to a bigraph.

\begin{definition}
\label{definition:laman_number}
 Let $B$ be a bigraph. If $f_B$ is dominant, we define the \emph{Laman number} 
of~$B$,
$\Lam(B)$, as~${\deg}{\left( f_B \right)}$, which can hence be either a 
positive number, or~$\infty$. 
Otherwise we set $\Lam(B)$ to zero.
\end{definition}

\begin{remark}
\label{remark:finite_laman_number}
 Notice that if $B$ is pseudo-Laman and $\Lam(B) > 0$, then $\Lam(B) \in 
\N\setminus\{0\}$. 
\end{remark}

If a bigraph has a self-loop or it is particularly simple, then its Laman 
number is zero or one, as shown by the following proposition.

\begin{proposition}
\label{proposition:base_cases}
 Let $B = (G,H)$ be a bigraph.
\begin{itemize}
 \item
  If $G$ or $H$ has a self-loop, then $\Lam(B) = 0$.
 \item
  If both $G$ and $H$ consist of a single edge that joins two vertices, then 
$\Lam(B) = 1$.
\end{itemize}
\end{proposition}
\begin{proof}
 If $G$ or $H$ has a self-loop, a direct inspection of the map~$f_B$ shows that 
the defining polynomial corresponding to the self-loop is 
zero, hence $f_B$ cannot be dominant.
If both $G$ and $H$ consist of a single edge that joins two vertices, then the 
map~$f_B$ reduces to the map $\p^0_{\C} \times \p^0_{\C} \longrightarrow 
\p^0_{\C}$, which has degree~$1$.
\end{proof}

By simply unraveling the definitions, we see that the number of realizations of 
a Laman graph, up to equivalence, can be expressed as a Laman number.

\begin{proposition}
\label{proposition:meaning_laman_number}
 Let $G$ be a Laman graph, then the Laman number of the bigraph~$(G,G)$ 
--- where biedges are the edges of~$G$ --- is equal to
the number of different realizations compatible with a general labeling of~$G$, 
up to direct complex isometries.
\end{proposition}

Due to \Cref{proposition:meaning_laman_number}, the problem we want 
to address in this work is a special instance of the problem of computing 
the Laman number of a bigraph. Notice, however, that the Laman number of an 
arbitrary bigraph does not have an immediate geometric interpretation.

\begin{remark}
\label{remark:choice_fiber}
  Let $B$ be a bigraph with biedges~$\edg$ such that $\Lam(B)>0$ and fix a 
biedge $\bar{e} \in \edg$. Since $f_{B}$ is a rational dominant map between 
varieties over~$\C$, there is a Zariski open subset $\mcal{U} \subseteq 
\p_{\C}^{|\edg|-1}$ such that the preimage of any point $p \in \mcal{U}$ 
under~$f_B$ consists of~$\Lam(B)$ distinct points. In particular, we can 
suppose 
that $p$ is of the form~$(\lambda_e)_{e \in \edg}$ with $\lambda_{\bar{e}} = 1$ 
and $(\lambda_e)_{e \in \edg \setminus \{ \bar{e} \}}$ a general point 
of~$\C^{\edg \setminus \{ \bar{e} \}}$.
\end{remark}

In the following we find it
useful to work in an affine setting: this is why in 
\Cref{definition:affine_sols} we introduce the sets~$Z^{B}_{\C}$. We
are going to use the language of affine schemes, mainly to be able to 
manipulate the equations freely without being concerned about the reducedness of
the ideal they generate. The reader not acquainted with scheme theory can 
harmlessly think about classical affine varieties, and indeed we are going to 
prove that the ideals we are concerned with are reduced. We 
first need to set some notation.

\begin{definition}
\label{definition:pairs}
Let $B = (G,H)$ be a bigraph, where $G = (V, \edg)$ and $H = (W, \edg)$. 
Define
\begin{align*}
 P & :=  \bigl\{ (u,v) \in V^2 \, : \, \{u,v\} \in \tau_G(\edg), \; u \neq v 
\bigr\}, \\
 Q & :=  \bigl\{ (t,w) \in W^2 \, : \, \{t,w\} \in \tau_H(\edg), \; t \neq w 
\bigr\}.
\end{align*}
Notice that the elements of~$P$ and $Q$ are ordered pairs 
(and this is conveyed also by the different notation used). In particular, from 
the definition we see that if $(u,v) \in P$, then also $(v,u) \in P$, and 
similarly for~$Q$. Moreover, we require the two elements in each pair to be 
different, and this is crucial in view of \Cref{definition:distances}.
\end{definition}

\begin{definition}
\label{definition:affine_sols}
Let $B = (G,H)$ be a bigraph with biedges~$\edg$ without self-loops.
Fix a biedge $\bar{e} \in \edg$. For a general point $(\lambda_e)_{e \in \edg 
\setminus \{ \bar{e} \}}$ in~$\C^{\edg \setminus \{ \bar{e} \}}$, we define 
$Z^B_{\C}$
as the subscheme of~$\C^P \times \C^Q$ defined by
\[
  \begin{cases}
		x_{\bar{u} \bar{v}} = y_{\bar{t} \bar{w}} = 1, 
		& \bar{u}\prec \bar{v},\ \bar{t}\prec \bar{w},
		\\[2pt]
		x_{uv} \tth y_{tw} = \lambda_e, 
		& \text{for all } e \in \edg \setminus \{ \bar{e} \}, 
			\, u\prec v, \, t\prec w,
		\\[2pt]
		\sum_{\mscr{C}} \, x_{uv} = 0,
		& \text{for all cycles } \mscr{C} \text{ in } G,
		\\[2pt]
		\sum_{\mscr{D}} \, y_{tw} = 0,
		& \text{for all cycles } \mscr{D} \text{ in } H,
  \end{cases}
\]
where we take $( x_{uv} )_{(u,v)\in P}$ and $( y_{tw} )_{(t,w) \in Q}$ as 
coordinates
and where
\begin{align*}
  \{\bar{u}, \bar{v}\}&=\tau_G(\bar{e}),&\{u,v\}&=\tau_G(e),&
  \{\bar{t}, \bar{w}\}&=\tau_H(\bar{e}),&\{t,w\}&=\tau_H(e).
\end{align*}
Here 
and in the following, when we write $\sum_{\mscr{C}} \, x_{uv}$ for a cycle 
$\mscr{C} = (u_0, u_1, \dotsc, u_n=u_0)$ in~$G$ we mean the expression
$x_{u_0 u_1} + \dotsb + x_{u_{n-1} u_0}$ (and similarly for cycles in~$H$). 
Notice that in cycles we allow repetitions of edges. In particular, if $(u,v) 
\in P$, one can 
always consider the cycle $(u,v,u)$, which implies the relation $x_{uv} = - 
x_{vu}$.
We drop the dependence of~$Z_{\C}^{B}$ on~$\bar{e}$ and~$(\lambda_{e})_{e \in 
\edg \setminus \{ \bar{e} \}}$ in the notation, since in the following it 
is clear from the context.
\end{definition}

\begin{example}
\label{example:Z_C}
 Consider the bigraph $(G,G)$ with set of biedges~$\edg$ as in 
\Cref{figure:4_laman}, that consists of two copies of the only Laman 
graph 
with $4$ vertices. Fix the biedge~$\bar{e}$ to be the one associated to 
the two edges connecting~$2$ and~$3$. If $(\lambda_e)_{e 
\in \edg 
\setminus \{ \bar{e} \}}$ is a general point, then the scheme~$Z^B_\C$ is 
defined by the following 
equations:
\begin{gather*}
x_{23} = y_{23}=1, \\
\begin{aligned}
x_{12}\tth y_{12} &= \lambda_\colcA,& 
x_{12}+x_{21}=x_{13}+x_{31}=x_{23}+x_{32}=x_{24}+x_{42}=x_{34}+x_{43}=0, \\
x_{13}\tth y_{13} &= \lambda_\colcB,& 
y_{12}+y_{21}=y_{13}+y_{31}=y_{23}+y_{32}=y_{24}+y_{42}=y_{34}+y_{43}=0, \\
x_{24}\tth y_{24} &= \lambda_\colcD,& 
x_{12}+x_{23}+x_{31}=y_{12}+y_{23}+y_{31}=0,                 \\
x_{34}\tth y_{34} &= \lambda_\colcE,& 
x_{24}+x_{43}+x_{32}=y_{24}+y_{43}+y_{32}=0. \\
\end{aligned}
\end{gather*}
Note that we did not include redundant equations coming from cycles such as 
$(1, 2, 4, 3, 1)$.
\end{example}

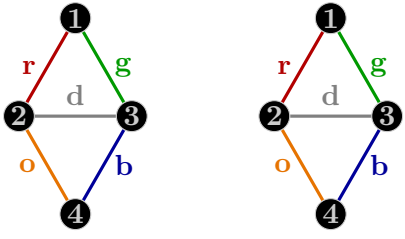
\begin{figure}
\begin{center}
\begin{tabular}{ccc}
\begin{tikzpicture}[scale=0.75]
\begin{scope}
	\lvertex (2) at (0,0) {2};
	\lvertex (3) at (2,0) {3};
	\lvertex (1) at (1,1.73205) {1};
	\lvertex (4) at (1,-1.73205) {4};
	
	\draw[edge1] (1) to node[lab1,left] {\colcA} (2);
	\draw[edge2] (1) to node[lab2,right] {\colcB} (3);
	\draw[edge3] (2) to node[lab3,above] {\colcC} (3);
	\draw[edge4] (2) to node[lab4,left] {\colcD} (4);
	\draw[edge5] (3) to node[lab5,right] {\colcE} (4);
\end{scope}
\end{tikzpicture}
& \qquad \qquad &
\begin{tikzpicture}[scale=0.75]
\begin{scope}
	\lvertex (2) at (0,0) {2};
	\lvertex (3) at (2,0) {3};
	\lvertex (1) at (1,1.73205) {1};
	\lvertex (4) at (1,-1.73205) {4};
	
	\draw[edge1] (1) to node[lab1,left] {\colcA} (2);
	\draw[edge2] (1) to node[lab2,right] {\colcB} (3);
	\draw[edge3] (2) to node[lab3,above] {\colcC} (3);
	\draw[edge4] (2) to node[lab4,left] {\colcD} (4);
	\draw[edge5] (3) to node[lab5,right] {\colcE} (4);
\end{scope}
\end{tikzpicture}
\end{tabular}
\end{center}
\caption{A bigraph that consists of two copies of the only Laman graph 
with $4$ vertices. Edges on the left and on the right bearing the same label 
are 
associated to the same biedge.}
\label{figure:4_laman}
\end{figure}

In the following lemma we show that the sets~$Z^{B}_{\C}$ can be used to 
compute the degree of~$f_B$.

\begin{lemma}
\label{lemma:passing_to_affine}
 Let $B = (G,H)$ be a bigraph with biedges~$\edg$ without self-loops.
Fix a biedge $\bar{e} \in \edg$. Let $p \in \p_\C^{|\edg|-1}$ be given by 
$p_{\bar{e}} = 1$ and $p_{e} = \lambda_e$ for all $e \in \edg \setminus 
\{\bar{e}\}$. Then the schemes~$f_B^{-1}(p)$ and~$Z_{\C}^{B}$ are isomorphic. 
In 
particular, $Z_{\C}^{B}$ consists of~$\Lam(B)$ distinct points.
\end{lemma}
\begin{proof}
 Write $\tau_G(\bar{e}) = \{\bar{u}, \bar{v}\}$ with $\bar{u} \prec \bar{v}$ 
and $\tau_H(\bar{e}) = \{\bar{t}, \bar{w}\}$ with $\bar{t} \prec \bar{w}$. We 
define a morphism from~$f_B^{-1}(p)$ to~$Z_{\C}^{B}$ by sending a point 
\[
 \bigl( [(x_v)_{v \in V}], [(y_w)_{w \in W}] \bigr) \in f_B^{-1}(p)
\]
to the point whose $uv$\hbox{-}coordinate is~$(x_u - x_v) / (x_{\bar{u}} - 
x_{\bar{v}})$, where $u \prec v$, for all $(u, v) \in P$, and whose 
$tw$\hbox{-}coordinate is~$(y_t - y_w) / (y_{\bar{t}} - y_{\bar{w}})$, where 
$t \prec w$, for all $(t,w) \in Q$.

We define a morphism from~$Z_{\C}^{B}$ to~$f_B^{-1}(p)$ as follows. For every 
component~$C$ of~$G$, fix a rooted spanning tree~$T_{C}$ and denote its 
root by~$r(C)$; similarly for~$H$. We send a point $\bigl( ( x_{uv} 
)_{(u,v) \in P} , 
( y_{tw} )_{(t,w) \in Q} \bigr) \in Z_{\C}^{B}$ to the point $\bigl([(x_u)_{u 
\in V}], [(y_t)_{t \in W}]\bigr) \in f^{-1}_{B}(p)$ such that if a vertex $u 
\in 
V$ belongs to the connected component~$C$, then $x_u = \sum_{i=0}^{n-1} x_{u_i 
u_{i+1}}$, where $(r(C)=u_0, \dotsc, u_n=u)$ is the unique path in~$T_C$ 
from~$r(C)$ to~$u$, and similarly for the vertices $t \in W$.
A direct computation shows that both maps are well-defined, and are each 
other's inverse. From this the statement follows.
\end{proof}

We conclude the section by proving a few results about the Laman number of a 
special kind of bigraph, that are used in \Cref{formulas} to 
obtain the final algorithm.

\begin{definition}
\label{definition:bridge}
Let $G$ be a graph and let $e$ be an edge of~$G$. We say that $e$ is a {\em 
bridge} if removing~$e$ increases the number of connected components of~$G$.
\end{definition}

\begin{lemma}
\label{lemma:two_bridges}
 Let $B = (G, H)$ be a pseudo-Laman bigraph with biedges~$\edg$ without 
self-loops and fix~$\bar{e} \in \edg$. If $\bar{e}$ is a bridge in both 
$G$ and~$H$, then $\Lam(B) = 0$.
\end{lemma}
\begin{proof}
 Suppose for a contradiction $\Lam(B) > 0$. Consider the equations 
defining~$Z^{B}_{\C}$. Since $\bar{e}$ is a bridge in 
both $G$ and~$H$, the variables $x_{\bar{u} \bar{v}}$ and $y_{\bar{t} 
\bar{w}}$, where $\{\bar{u}, \bar{v}\} = \tau_G(\bar{e})$ and $\{\bar{t}, 
\bar{w}\} = \tau_H(\bar{e})$, do not appear in any of the equations defined by 
cycles in~$G$ or in~$H$ except for the equations 
$x_{\bar{u} \bar{v}} = -x_{\bar{v} \bar{u}}$ 
and 
$y_{\bar{t} \bar{w}} = -y_{\bar{w} \bar{t}}$.
Hence, the system of equations
\[
  \begin{cases}
   x_{uv} \tth y_{tw} = \lambda_e, &
    \text{for all } e \in \edg \setminus \{ \bar{e} \}, \, u\prec v,
    \, t\prec w, \\[2pt]
   \sum_{\mscr{C}} \, x_{uv} = 0,  &
		\text{for all cycles } \mscr{C}
		\text{ in } G \bigsubtract \bigl\{ \bar{e} \bigr\}, \\[2pt]
   \sum_{\mscr{D}} \, y_{tw} = 0, &
		\text{for all cycles } \mscr{D}
		\text{ in } H \bigsubtract \bigl\{ \bar{e} \bigr\}
  \end{cases}
\]
defines an affine 
scheme $\widetilde{Z}$ isomorphic to~$Z^{B}_{\C}$. One notices, however, 
that if $(x_{uv}, y_{tw})$ is a point in $\widetilde{Z}$, then for every $\eta
\in \C\setminus\{0\}$ also the point $(\eta \tth x_{uv}, 
\frac{1}{\eta} y_{tw})$ is in $\widetilde{Z}$. This implies that $Z^{B}_{\C}$ 
has infinite cardinality, which contradicts the pseudo-Laman assumption on~$B$.
\end{proof}

\begin{lemma}
\label{lemma:remove_edge}
 Let $B = (G, H)$ be a pseudo-Laman bigraph with biedges~$\edg$ without 
self-loops and fix~$\bar{e} \in \edg$. If $\bar{e}$ is a bridge in~$G$, 
but not in~$H$, then 
\[
  \Lam(B) \, = \,
  \Lam\bigl( (G \subtract \{ \bar{e} \}, 
              H \subtract \{ \bar{e} \})
       \bigr).
\]
\end{lemma}
\begin{proof}
Consider another biedge~$\tilde{e}$ and 
use it to define the scheme~$Z^{B}_{\C}$. Its equations are:
\[
  Z^{B}_{\C} \colon
  \begin{cases}
   x_{\tilde{u} \tilde{v}} = y_{\tilde{t} \tilde{w}} = 1, &
    \tilde{u}\prec \tilde{v},\ \tilde{t}\prec \tilde{w}, \\[2pt]
   x_{uv} \tth y_{tw} = \lambda_e, &
    \text{for all } e \in \edg \setminus \{ \tilde{e} \},
    \, u\prec v, \, t\prec w, \\[2pt]
   \sum_{\mscr{C}} \, x_{uv} = 0,  &
		\text{for all cycles } \mscr{C} \text{ in } G,  \\[2pt]
   \sum_{\mscr{D}} \, y_{tw} = 0, &
		\text{for all cycles } \mscr{D} \text{ in } H.
  \end{cases}
\]
Now consider the bigraph $\widetilde{B} = \bigl( G \subtract 
\{ \bar{e} \}, H \subtract \{ \bar{e} \} \bigr)$.
Notice that we can still use~$\tilde{e}$ to define the 
scheme~$Z^{\widetilde{B}}_{\C}$. Its equations are:
\[
  Z^{\widetilde{B}}_{\C} \colon
  \begin{cases}
   x_{\tilde{u} \tilde{v}} = y_{\tilde{t} \tilde{w}} = 1, &
    \tilde{u}\prec \tilde{v},\ \tilde{t}\prec \tilde{w}, \\[2pt]
   x_{uv} \tth y_{tw} = \lambda_e, &
    \text{for all } e \in \edg \setminus \{ \bar{e}, \tilde{e} \},
    \, u\prec v, \, t\prec w, \\[2pt]
   \sum_{\mscr{C}} \, x_{uv} = 0,  &
		\text{for all cycles } \mscr{C}
		\text{ in } G \bigsubtract \bigl\{ \bar{e} \bigr\}, \\[2pt]
   \sum_{\mscr{D}} \, y_{tw} = 0, &
		\text{for all cycles } \mscr{D}
		\text{ in } H \bigsubtract \bigl\{ \bar{e} \bigr\}.
  \end{cases}
\]
We are going to prove that $Z^{B}_{\C}$ and $Z^{\widetilde{B}}_{\C}$ are 
isomorphic, concluding the proof. Since $\bar{e}$ is a bridge in~$G$, 
the coordinate $x_{\bar{u} \bar{v}}$ appears in the equations 
of~$Z^{B}_{\C}$ only in $x_{\bar{u} \bar{v}} \tth y_{\bar{t} \bar{w}} = 
\lambda_{\bar{e}}$, and in 
$x_{\bar{u} \bar{v}} = -x_{\bar{v} \bar{u}}$. 
This means that the image of~$Z^{B}_{\C}$ under the projection from 
the coordinates $x_{\bar{u} \bar{v}}$, $x_{\bar{v} \bar{u}}$, $y_{\bar{t}, 
\bar{w}}$ and $y_{\bar{w}, \bar{t}}$ coincides 
with~$Z^{\widetilde{B}}_{\C}$. Moreover, the projection is an isomorphism 
on~$Z^{B}_{\C}$: in fact,  
the $y_{\bar{t}, \bar{w}}$\hbox{-}coordinate can be recovered by a 
cycle condition (recall that $\bar{e}$ is not a bridge in~$H$, so it 
appears in a cycle different from the trivial cycle $(\bar{t}, \bar{w}, 
\bar{t})$). Then the $x_{\bar{u}, \bar{v}}$\hbox{-}coordinate can be 
recovered from the equation $x_{\bar{u} \bar{v}} \tth y_{\bar{t} \bar{w}} = 
\lambda_{\bar{e}}$.
\end{proof}

\begin{definition}
\label{definition:untangles}
Let $B = (G, H)$ be a bigraph with biedges~$\edg$  without self-loops and let 
$\bar{e}\in\edg$ be fixed. Suppose that the graph $G$ splits into  disconnected 
subgraphs $G_1'$, $G_2'$ and that $H$ splits into disconnected subgraphs 
$H_1'$, 
$H_2'$. Suppose further that $\edg=\edg_1\cup\edg_2\cup\{\bar{e}\}$ decomposes 
into three disjoint subsets such that 
\[
 G_1'=(V_1',\edg_1\cup \{\bar{e}\}),\quad 
 G_2'=(V_2',\edg_2)
 \quad\textrm{and}\quad
 H_1'=(W_1',\edg_1),\quad
 H_2'=(W_2',\edg_2\cup \{\bar{e}\}).
\]
Under these assumptions we say that the bigraph $B$ \emph{untangles} via 
$\bar{e}$ 
into bigraphs
\begin{align*}
B_1& :=
\bigl(
G_1' \subtract \{ \bar{e} \},~
H_1'
\bigr),
&
B_2&:=
\bigl(
G_2',~
H_2' \subtract \{ \bar{e} \}
\bigr).
\end{align*}
See \Cref{figure:quotient_bidistance} for an example of a bigraph that 
untangles via an edge (the gray vertical one).
\end{definition}

\begin{proposition}
\label{proposition:splitting}
Suppose that a bigraph $B = (G,H)$ with biedges~$\edg$ without 
self-loops untangles via $\bar e \in \edg$ into 
bigraphs $B_1$ and $B_2$, where $\bar{e}$ is neither a bridge in~$G$ 
nor in~$H$, then
\[
 \Lam(B) \, = \, \Lam(B_1) \cdot \Lam(B_2).
\]
\end{proposition}
\begin{proof}
We use the notation from \Cref{definition:untangles}. 
The hypothesis implies that 
\[
 \dim(G) \, = \, \dim(G_1') + \dim(G_2') \quad \text{ and } \quad
 \dim(H) \, = \, \dim(H_1') + \dim(H_2').
\]
Set $\{\bar{u}, \bar{v}\} = \tau_G(\bar{e})$ and $\{\bar{t}, \bar{w}\} = 
\tau_H(\bar{e})$.
Fix a biedge~$e_1 \in \edg_1$ and let
$\{t_1, w_1\} = \tau_H(e_1)$. Similarly, fix a biedge~$e_2 \in \edg_2$ 
and let $\{u_2, v_2\} = \tau_G(e_2)$.
We consider the following three rational maps:
\begin{gather*}
 \begin{array}{rcl}
  \p^{\dim(G)-1}_{\C} & \dashrightarrow & \p^{\dim(G_1')-1}_{\C} \times 
	\p^{\dim(G_2')-1}_{\C} \times \p^1_{\C} \\[2pt]
  ([(x_v)_{v \in V}]) & \longmapsto & \bigl( [(x_v)_{v \in V_1'}], [(x_v)_{v 
\in 
V_2'}],
		(x_{\bar{u}} - x_{\bar{v}}: x_{u_2} - x_{v_2}) \bigr)
	\\[8pt]
  \p^{\dim(H)-1}_{\C} & \dashrightarrow & \p^{\dim(H_1')-1}_{\C} \times 
	\p^{\dim(H_2')-1}_{\C} \times \p^1_{\C} \\[2pt]
  ([(y_w)_{w \in W}]) & \longmapsto & \bigl( [(y_w)_{w \in W_1'}], [(y_w)_{w 
\in 
W_2'}],
		(y_{\bar{t}} - y_{\bar{w}}: x_{t_1} - x_{w_1}) \bigr)
	\\[8pt]
  \p^{|\edg|-1}_{\C} & \dashrightarrow & \p^{|\edg_1|-1}_{\C} \times 
	\p^{|\edg_2|-1}_{\C} \times \p^1_{\C} \times \p^1_{\C} \\[2pt]
  (z_e)_{e \in \edg} & \longmapsto & \bigl( (z_e)_{e \in \edg_1}, (z_e)_{e \in 
\edg_2},
		(z_{\bar{e}}: z_{e_1}), (z_{\bar{e}}: z_{e_2}) \bigr) 
 \end{array}
\end{gather*}
One can check that these maps are birational. We define the rational 
map~$\hat{f}$ so that the following diagram is commutative:
\[
 \xymatrix@C=2cm{\p^{\dim(G)-1}_{\C} \times \p^{\dim(H)-1}_{\C} \ar@{<-->}[r] 
\ar@{-->}[d]^{f_B} & \mbox{$
 \begin{array}{c}
  \left( \p^{\dim(G_1')-1}_{\C} \times \p^{\dim(G_2')-1}_{\C} \times \p^1_{\C} 
\right) \\ \bigtimes \\ 
  \left( \p^{\dim(H_1')-1}_{\C} \times \p^{\dim(H_2')-1}_{\C} \times \p^1_{\C} 
\right) 
 \end{array} $} \ar@{-->}[d]^-{\hat{f}} \\
\p^{|\edg|-1}_{\C} \ar@{<-->}[r] & \p^{|\edg_1|-1}_{\C} \times 
\p^{|\edg_2|-1}_{\C} \times \p^1_{\C} \times \p^1_{\C} } 
\]
It follows that ${\deg}{\left( f_B \right)} = \deg \bigl( \hat{f} \bigr)$.
Denote $[(x_v)_{v \in V_i'}]$ by~$[X_i]$ for $i \in \{1,2\}$, and denote 
$[(y_w)_{w \in W_i'}]$ by~$[Y_i]$ for $i \in \{1,2\}$. An explicit computation 
shows that $\hat{f}$ sends a point
\[
 \Bigl( [X_1], [X_2], (\mu_G: \nu_G) \Bigr), 
 \Bigl( [Y_1], [Y_2], (\mu_H: \nu_H) \Bigr)
\]
to the point
\[
 \Bigl( 
  \underbrace{f_{B_1} \bigl( [X_1],[Y_1] \bigr)}_{\in \p^{|\edg_1|-1}_{\C}}, 
  \underbrace{f_{B_2} \bigl( [X_2],[Y_2] \bigr)}_{\in \p^{|\edg_2|-1}_{\C}}, 
  \underbrace{\bigl( \mu_G\,\delta_G([X_1]):\nu_G\,\delta_G([X_1]) 
\bigr)}_{\in \p^1_{\C}},
  \underbrace{\bigl( \mu_H\,\delta_H([Y_2]):\nu_H\,\delta_H([Y_2]) 
\bigr)}_{\in \p^1_{\C}} 
 \Bigr),
\]
where $\delta_G \colon \p^{\dim(G_1')-1}_{\C} \dashrightarrow \C$ and 
$\delta_H \colon \p^{\dim(H_2')-1}_{\C} \dashrightarrow \C$ 
are some rational functions. 
From 
the explicit form of~$\hat{f}$ we see that $\deg \bigl( \hat{f} \bigr) = 
\deg \bigl( \hat{f}_1 \bigr) \cdot \deg \bigl( \hat{f}_2 \bigr)$, where
the map~$\hat{f}_1$ is given by
\[
 \begin{array}{rrcl}
  \p^{\dim(G_1')-1}_{\C} \times \p^{\dim(H_1')-1}_{\C}
  \times\p^1_{\C} & \dashrightarrow & \p^{|\edg_1|-1}_{\C} \times \p^1_{\C} 
\\[4pt]
  \bigl( [X_1], [Y_1], (\mu_G: \nu_G) \bigr) & \longmapsto &
  \!f_{B_1} \bigl([X_1],[Y_1] \bigr), \bigl( \mu_G\,\delta_G([X_1]): 
\nu_G\,\delta_G([X_1]) \bigr)\!
 \end{array}
\]
and similarly for~$\hat{f}_2$. Note that for both $i \in \{1, 2\}$, 
the map~$\hat{f}_i$ is the restriction to a suitable open set of 
the map~$f_{B_i} \times \mathrm{id}_{\p^1_{\C}}$, since the rational 
maps~$\delta_G$ and~$\delta_H$ do not have any other influence than restricting 
the domain of the map. This means that $\deg \bigl( \hat{f}_i \bigr) = {\deg}{ 
\left( f_{B_i} \right)}$ for both $i \in \{1, 2\}$, which concludes the proof.
\end{proof}

\begin{lemma}
\label{lemma:single_bridge}
If a pseudo-Laman bigraph $B = (G,H)$  without self-loops untangles via 
$\bar{e}\in\edg$ into bigraphs~$B_1$ and~$B_2$ such that $\bar{e}$ is a bridge 
in~$G$ but not in~$H$, then $\Lam(B) = 0$.
\end{lemma}
\begin{proof}
 By \Cref{lemma:remove_edge}, we know that $\Lam(B) = 
\Lam\bigl(\widetilde{B}\bigr)$, 
where $\widetilde{B}=\bigl( G \subtract \{ \bar{e} 
\}, H \subtract \{ \bar{e} \} \bigr)$. 
It follows that $\widetilde{B}$ is the disjoint union of $B_1$ and $B_2$.
Using the same technique adopted in \Cref{lemma:two_bridges} we see that 
if $\Lam\bigl(\widetilde{B}\bigr)$ were positive, then we could scale the 
points 
in 
$Z^{B_1}_{\C}$ by arbitrary scalars $\eta \in \C \setminus \{0\}$, obtaining an 
infinite Laman number. This would contradict the pseudo-Laman hypothesis, so 
the statement is proved.
\end{proof}

\begin{lemma}
\label{lemma:bridge_pseudo}
If a pseudo-Laman bigraph $B = (G,H)$ without self-loops untangles via 
$\bar{e}\in\edg$ into bigraphs~$B_1$ and~$B_2$, where $\bar{e}$ is a bridge 
in~$G$, then either~$B_1$ or~$B_2$ is not pseudo-Laman. 
\end{lemma}
\begin{proof}
Suppose that both $B_1$ and $B_2$ are pseudo-Laman; we show that this leads to
a contradiction. Using the hypothesis one sees that $\dim \bigl( G_1' \subtract 
\{ \bar{e} \} \bigr) = \dim(G_1') - 1$. Moreover, if 
$\bar{e}$ is also a bridge, then $\dim \bigl( H_2' \subtract 
\{ \bar{e} \} \bigr) = \dim(H_2') - 1$, otherwise we have $\dim 
\bigl( H_2' \subtract \{ \bar{e} \} \bigr) = \dim(H_2')$. 
Since $\dim(G) = \dim(G_1') + \dim(G_2')$ and $\dim(H) = \dim(H_1') 
+ \dim(H_2')$, the pseudo-Lamanity of~$B_1$ and~$B_2$ implies 
\[
 \dim(G_1') - 1 + \dim(G_2') + \dim(H_1') + \dim(H_2') \, \geq \, |\edg_1| + 
|\edg_2| +2,
\]
where $\edg_1$ and~$\edg_2$ are the biedges of~$B_1$ and~$B_2$, respectively 
(the inequality~$\geq$ takes into account the fact that $\bar{e}$ 
may or may not be a bridge). Since $|\edg| = |\edg_1| + |\edg_2| + 1$, the 
previous equation in turn implies
\[
 \dim(G) + \dim(H) \, \geq  \, |\edg| + 2,
\]
contradicting the hypothesis that $B$ is pseudo-Laman.
\end{proof}

\section{Bidistances and quotients}
\label{bidistances}

Tropical geometry is a technique that allows us to transform systems of 
polynomial 
equations into systems of piecewise linear equations. This is possible if one 
works over the field of Puiseux series. An algebraic relation between Puiseux 
series implies a piecewise linear relation between their orders (which are 
rational numbers). One hopes that the piecewise linear system is easier to 
solve; if so, one has candidates for the orders of solutions of the initial 
system, and sometimes this is enough to obtain the desired information. This 
technique has been successfully used, amongst others, by 
Mikhalkin~\cite{Mikhalkin2005} to count the number of algebraic curves with 
some prescribed properties.

We use a similar idea for computing the Laman number of a pseudo-Laman bigraph. 
As we pointed out in the Introduction, this amounts to compute the base degree 
of the algebraic matroid associated to the variety parametrizing distances 
between pairs of points; however, we do not use the matroid formalism in our 
work.  
For each pseudo-Laman bigraph~$B$, we need to know the number of solutions 
of the system defining~$Z^B_\C$. This number coincides with the number of 
solutions of a ``perturbed'' system over the Puiseux field 
(\Cref{lemma:degree_fiber}). 
The orders of each solution of the new system satisfy piecewise 
linear conditions (\Cref{definition:distances}). We prove 
(\Cref{lemma:equations_Z,lemma:induction,lemma:specialization}) that the
Puiseux series solutions sharing the same orders are in bijection with the
complex solutions of another system of equations of a certain ``quotient
bigraph''.  This yields a first recursive scheme
(\Cref{theorem:algebraic_counting}).

\begin{notation}
 Denote by~$\K$ the field $\C \{\!\{ s \}\!\}$ of Puiseux series with 
coefficients in~$\C$. Recall that $\K$ is of characteristic zero and is 
algebraically closed. The field~$\K$ is equipped with a valuation $\nu \colon 
\K \setminus \{0\} \longrightarrow \Q$ associating to an element 
$\sum_{i=k}^{+\infty} c_i \tth s^{i/n}$ the rational 
number~$k/n$, where $k\in\Z$ and $c_k\neq0$. Recall that $\nu(a \cdot b) = 
\nu(a) + \nu(b)$ and $\nu(a+b) \geq 
\min \{ \nu(a), \nu(b) \}$.
\end{notation}

\begin{definition}
 Let $B = (G,H)$ be a bigraph. Define $f_{B, \K}$ to be the map obtained as the 
extension 
of scalars, via the natural inclusion $\C \hookrightarrow \K$, of the rational
map~$f_{B}$ associated to~$B$ (see \Cref{definition:bigraph_map}). 
This means that, with the notation as in 
\Cref{definition:bigraph_space}, 
we define
\begin{align*}
 \p^{\dim(G) - 1}_{\K} & := \p \bigl( \K^V \bigquotient (L_G \otimes_{\C} \K)
\bigr),
&
 \p^{\dim(H) - 1}_{\K} & := \p \bigl( \K^W \bigquotient (L_H \otimes_{\C} \K)
\bigr),
\end{align*}
and then $f_{B,\K} \colon \p^{\dim(G) - 1}_{\K} \times \p^{\dim(H) - 1}_{\K} 
\dashrightarrow \p^{|\edg|-1}_{\K}$ is given by the same equations as~$f_{B}$.
\end{definition}

\begin{remark}
\label{remark:degree_extension}
By construction, ${\deg}{\left( f_{B} \right)}$ is defined if and only 
if ${\deg}{\left( f_{B,\K} \right)}$ is defined, and in that case they 
coincide. 
In fact, $f_{B}$ is dominant if and only if $f_{B, \K}$ is so. In this case, 
let 
$Y_{\C}$ be the open subset where $f_{B}$ is defined. Because $f_{B} \colon 
Y_{\C} \longrightarrow \p^{|\edg|-1}_{\C}$ is a dominant morphism between 
complex varieties, there exists an open subset $\mcal{U}_{\C} \subseteq 
\p^{|\edg|-1}_{\C}$ such that the fiber of~$f_{B}$ over any point 
of~$\mcal{U}_{\C}$ consists of~${\deg}{\left( f_B \right)}$ distinct points. 
Since $f_{B, \K}$ is the extension of scalars of~$f_B$, it follows that also 
the 
fiber of~$f_{B, \K}$ over any point in $\mcal{U}_{\K} := \mcal{U}_{\C} 
\times_{\Spec(\C)} \Spec(\K)$ consists of ${\deg}{\left( f_B \right)}$ 
distinct points. 
In fact, for every $q_{_{\K}} \in \mcal{U}_{\K}$ we have $f_{B, 
\K}^{-1}(q_{_{\K}}) \cong f_{B}^{-1}(q_{_{\C}}) \times_{\Spec(\C)} 
\Spec{(\K)}$, 
where $q_{_{\C}}$ is the image of~$q_{_{\K}}$ under the natural morphism 
$\mcal{U}_{\K} \longrightarrow \mcal{U}_{\C}$. Hence the cardinality of~$f_{B, 
\K}^{-1}(q_{_{\K}})$ is equal to the cardinality 
of~$f_B^{-1}(q_{_{\C}})$ and therefore ${\deg}{\left( f_{B,\K} \right)} =
{\deg}{\left( f_B \right)}$.
\end{remark}

The fact that the map~$f_{B, \K}$ is defined over~$\C$, and not over~$\K$, 
gives us a lot of freedom concerning the valuation of the general point whose
fiber we consider. More precisely:

\begin{lemma}
\label{lemma:degree_fiber}
 Let $B$ be a bigraph such that $\Lam(B)>0$. Fix a vector $\wt = \bigl( \wt(e) 
\bigr)_{e \in \edg} \in \Q^{\edg}$. Then ${\deg}{\left( f_{B, \K} \right)}$ 
coincides with the cardinality of the fiber of~$f_{B, \K}$ over any point $p 
\in 
\p^{|\edg|-1}_{\K}$ of the form $p = \bigl( \lambda_e \tth s^{\wt(e)} \bigr)_{e 
\in \edg}$, where $(\lambda_e)_{e \in \edg}$ is a general point in~$\C^{\edg}$. 
\end{lemma}
\begin{proof}
 Consider the rational map $f_B \colon \p^{\dim(G) - 1}_{\C} \times \p^{\dim(H) 
- 1}_{\C} \dashrightarrow \p^{|\edg|-1}_{\C}$, which is dominant by 
hypothesis. To prove the statement it is enough to show that a 
point~$p$ satisfying the hypothesis lies in the set~$\mcal{U}_{\K}$ defined in 
\Cref{remark:degree_extension}. Suppose by contradiction that $p \not\in 
\mcal{U}_{\K}$. Since $\mcal{U}_{\K}$ is Zariski open, it is 
defined by a disjunction of polynomial inequalities with coefficients in~$\C$. 
Let $g \neq 0$ be one of 
these inequalities: by assumption $g(p) = 0$, but this implies that 
$\tilde{g}\bigl( (\lambda_e)_{e \in \edg} \bigr) = 0$ for some non-zero 
polynomial~$\tilde{g}$ over~$\C$, contradicting the generality 
of~$(\lambda_e)_{e \in \edg}$.
\end{proof}

Let $B$ be a bigraph such that $\Lam(B)>0$.
Fix a vector $\wt = \bigl( \wt(e) \bigr)_{e \in \edg} \in \Q^{\edg}$ and
a biedge $\bar{e} \in \edg$.
Arguing as in \Cref{remark:choice_fiber}, we see that it is 
enough to consider fibers of~$f_{B, \K}$ over points~$p$ of the form 
$p_{\bar{e}} = 1$, while $p_e = \lambda_e \tth s^{\wt(e)}$ for a general 
point $(\lambda_e)_{e \in \edg \setminus \{ \bar{e} \}}$ in $\C^{\edg \setminus 
\{ \bar{e} \}}$.
This is why we formulate the following assumption, which is used throughout 
this section.

\begin{assumption}\label{assumption:general}
	Let $B$ be a pseudo-Laman bigraph with biedges~$\edg$ such that $\Lam(B)>0$. 
Notice that by \Cref{proposition:base_cases} this implies 
that $B$ has no self-loops.
	Fix a biedge~$\bar{e} \in \edg$, fix 
	$\wt \in \Q^{\edg \setminus \{\bar{e}\}}$ and let $(\lambda_e)_{e \in 
	\edg \setminus\{\bar{e}\}}$ be a general point in~$\C^{\edg \setminus 
	\{\bar{e}\}}$. Let $p \in \p^{|\edg|-1}_{\K}$ be such that $p_{\bar{e}} = 1$ 
	and $p_e = \lambda_e \tth s^{\wt(e)}$ for all biedges~$e \in \edg 
	\setminus\{\bar{e}\}$.
\end{assumption}

\begin{remark}
\label{remark:passing_to_affine_puiseux}
 Let $B = (G,H)$ be a bigraph with biedges~$\edg$ and use 
\Cref{assumption:general}.
 Following \Cref{lemma:passing_to_affine} one can prove that $f^{-1}_{B, 
\K}(p)$ is isomorphic to 
\[
 Z^B_{\K} \, := \, \mathrm{Spec} 
 \left(\!
  \begin{array}{ll}
   x_{\bar{u} \bar{v}} = y_{\bar{t} \bar{w}} = 1 &
    \bar{u}\prec \bar{v}, \, \bar{t}\prec \bar{w} \\[2pt]
   x_{uv} \tth y_{tw} = \lambda_e \tth s^{\wt(e)} &
    \text{for all } e \in \edg \setminus \{ \bar{e} \}, \, u\prec v, \, t\prec 
w 
\\[2pt]
  \sum_{\mscr{C}} \, x_{uv} = 0  & \text{for all cycles } \mscr{C} 
\text{ in } G \\[2pt]
   \sum_{\mscr{D}} \, y_{tw} = 0 & \text{for all cycles } \mscr{D} 
\text{ in } H
  \end{array}\!
 \right) \subseteq \K^{P} \times \K^{Q},
\]
where the notation is as in \Cref{definition:affine_sols}.
\end{remark}

\begin{example}
\label{example:Z_K}
 We continue with \Cref{example:Z_C}: if we fix the vector~$\wt$ to be $( 1 
)_{\edg \setminus \{ \bar{e} \}}$, then the scheme~$Z^{B}_{\K}$ is defined by 
the equations
\begin{align*}
	x_{23} &= 1, & x_{12}\tth y_{12} &= \lambda_\colcA\tth s, & x_{24}\tth y_{24} 
&= \lambda_\colcD\tth s, \\
	y_{23} &= 1, & x_{13}\tth y_{13} &= \lambda_\colcB\tth s, & x_{34}\tth y_{34} 
&= \lambda_\colcE\tth s,
\end{align*}
and by the equations coming from the cycles (they are the same as in 
\Cref{example:Z_C}).
\end{example}

If $p$ is a point of the form $\bigl(\lambda_e \tth s^{\wt(e)}\bigr)_{e \in 
\edg}$ in the codomain of the map~$f_{B, \K}$, then for every point $q \in 
f^{-1}_{B, \K}(p)$ we can consider the vector of the valuations of its 
coordinates. In terms of tropical geometry, this means that we take the 
tropicalization of the preimage~$f^{-1}_{B, \K}(p)$. In 
\Cref{definition:distances} we associate 
to each such point~$q$ a discrete object, which we call \emph{bidistance} 
(see \Cref{definition:bidistance}). We then partition the set~$f^{-1}_{B, 
\K}(p)$ according to the bidistances that are determined by its points.

\begin{definition}
\label{definition:distances}
 Let $B = (G,H)$ be a bigraph with biedges~$\edg$ and use 
\Cref{assumption:general}. Fix 
$q \in f_{B, \K}^{-1}(p)$. Then $q 
= \bigl( [(x_v)_{v \in V}], [(y_w)_{w \in W}] \bigr)$ and by construction
\begin{gather*}
 \frac{x_{u} - x_{v}}{x_{\bar{u}} - x_{\bar{v}}} \cdot 
\frac{y_{t} - y_{w}}{y_{\bar{t}} - y_{\bar{w}}} \, = 
\, \lambda_{e} \tth s^{\wt(e)} \quad \text{for all } e \in 
\edg \setminus\{\bar{e}\}, \text{ where} \\
 \begin{array}{ll}
  \{\bar{u}, \bar{v}\}=\tau_G(\bar{e}), & \bar{u}\prec\bar{v}, \\
  \{\bar{t}, \bar{w}\}=\tau_H(\bar{e}), & \bar{t}\prec\bar{w},
 \end{array}
 \quad \text{and} \quad 
 \begin{array}{ll}
  \{u,v\}=\tau_G(e), & u\prec v, \\
  \{t,w\}=\tau_H(e), & t\prec w.
 \end{array}
\end{gather*}
We define two functions $d_V \colon P \longrightarrow \Q$ and
$d_W \colon Q \longrightarrow \Q$, with $P$ and~$Q$ 
as in \Cref{definition:pairs}:
\begin{alignat*}2
 d_V(u,v) & := \nu \biggl( \frac{x_{u} - x_{v}}{x_{\bar{u}} - x_{\bar{v}}} 
\biggr)
 &\quad & \text{for all } (u,v) \in P, \\
 d_W(t,w) & := \nu \biggl( \frac{y_{t} - y_{w}}{y_{\bar{t}} - y_{\bar{w}}} 
\biggr)
 && \text{for all } (t,w) \in Q.
\end{alignat*}
Notice that the definition of~$P$ and~$Q$ implies that
$x_{u} - x_{v}$ and $y_{t} - y_{w}$ are always nonzero.
Moreover, both $d_V$ 
and $d_W$ depend on~$q$, but not on the representatives $(x_v)_{v \in V}$ 
and $(y_w)_{w \in W}$.
\end{definition}

\begin{lemma}
\label{lemma:properties_distance}
With the notation and assumptions as in \Cref{definition:distances}, the two
functions $d_V \colon P \longrightarrow \Q$ and $d_W \colon Q \longrightarrow 
\Q$ satisfy:
\begin{itemize}
 \item $d_V(u,v) = d_V(v,u)$ for all $(u,v) \in P$, and similarly for~$d_W$;
 \item $d_V(u,v) + d_W(t,w) = \wt(e)$ for all $e \in \edg \setminus \{ \bar{e} 
\}$,
  where $\{u, v\}=\tau_G(e)$ and $\{t, w\}=\tau_H(e)$;
 \item $d_V(\bar{u}, \bar{v}) = d_W(\bar{t}, \bar{w}) = 0$, where $\{\bar{u}, 
  \bar{v}\}=\tau_G(\bar{e})$ and $\{\bar{t}, \bar{w}\}=\tau_H(\bar{e})$;
 \item for every cycle~$\mscr{C}$ in~$G$, the minimum of the values of~$d_V$ on 
the pairs of vertices~$(u,v)$ appearing in~$\mscr{C}$ is attained at least 
twice, and similarly for~$d_W$.
\end{itemize}
\end{lemma}
\begin{proof}
 The statement follows from the definitions and the properties of the 
valuation, see \cite[Section VI.3.1, Definition~1 and Corollary 
to Proposition~1]{Bourbaki1972}. In particular, we use that $\nu(a) = 
\nu(-a)$ and $\nu(a \cdot b) = \nu(a) + \nu(b)$ for all nonzero $a$ and~$b$.
The fourth property 
follows from $\sum_{\mscr{C}} \, (x_u - x_v)/(x_{\bar{u}} - x_{\bar{v}}) = 0$
if $\mscr{C}$ is a cycle in~$G$ (and 
similarly for cycles in~$H$): we employ the fact that if the sum 
of finitely many elements is zero, then the minimum of their 
valuations is achieved at least twice. Notice 
that the values $d_V(\bar{u}, \bar{v})$ and $d_W(\bar{t}, \bar{w})$ are defined 
because $\bar{e}$ is not a self-loop by \Cref{assumption:general}.
\end{proof}

\begin{definition}
\label{definition:bidistance}
 Let $B$ be a bigraph with biedges~$\edg$ without self-loops, let $\bar{e}$ 
be a fixed biedge, and 
let $\wt \in \Q^{\edg \setminus \{\bar{e}\}}$. A \emph{bidistance}~$d$ on~$B$
compatible with~$\wt$ is a pair $(d_V, d_W)$ of functions $d_V \colon P 
\longrightarrow \Q$ and $d_W \colon Q \longrightarrow \Q$ such that the 
conditions of \Cref{lemma:properties_distance} are satisfied. If the
weight vector is clear from the context, we omit the clause
``compatible with~$\wt$''.
\end{definition}

\begin{remark}
 Let $B$ be a bigraph and use \Cref{assumption:general}. Then any
$q \in f^{-1}_{B, \K}(p)$ defines a bidistance~$d$ on~$B$, and via the 
isomorphism provided by \Cref{remark:passing_to_affine_puiseux} also any 
point in~$Z_\K^{B}$ defines a bidistance.
\end{remark}

As mentioned before, we are going to count the number of points in a general 
fiber of~$f_{B, \K}$ that determine a fixed bidistance.
We do so by computing the Laman 
number of a ``smaller'' bigraph, obtained via a quotient 
operation as explained in \Cref{definition:quotient_bigraph}.

\begin{definition}
\label{definition:quotient_bigraph}
 Let $B=(G,H)$ be a bigraph with set of biedges~$\edg$ and without self-loops, 
and 
fix a bidistance $d = (d_V, d_W)$ on~$B$. We define a new bigraph $B_d$ as 
follows:
For every $\alpha 
\in \operatorname{im}(d_V)$, define the graphs $G_{\geq 
\alpha}$ and $G_{> \alpha}$ to be the subgraphs of~$G$ determined by all edges 
with endpoints~$u$ and~$v$ such that $d_V(u,v) \geq \alpha$ and $d_V(u,v) > 
\alpha$, respectively. Similarly, for every $\beta \in \operatorname{im}(d_W)$, 
define~$H_{\geq \beta}$ and~$H_{> \beta}$. Let
\[
  G_{d_V} \, := \, \bigdotcup{\alpha \in \operatorname{im}(d_V)}{} G_{\geq 
\alpha} \bigquotient G_{> \alpha}
  \qquad\text{and}\qquad
  H_{d_W} \, := \, \bigdotcup{\beta \in \operatorname{im}(d_W)}{} 
   H_{\geq \beta} \bigquotient H_{> \beta} \ .
\]
Here by $G_{\geq \alpha} \bigquotient G_{> \alpha}$ and $H_{\geq \beta} 
\bigquotient H_{> \beta}$ we mean the quotients of graphs as described in 
\Cref{definition:quotient_graphs}, followed by removing singleton 
components without edges. 
The union symbol~$\dotcup$ indicates the disjoint union of graphs. 

There is a natural bijection between edges of~$G$ and edges of~$G_{d_V}$, 
sending each edge~$e$ in~$G$ to the corresponding edge in the quotient $G_{\geq 
d_V(\tau_G(e))} \bigquotient G_{> d_V(\tau_G(e))}$.
We define $B_{d}$ to be the bigraph $(G_{d_V}, H_{d_W})$ with set of 
biedges~$\edg$ inherited from~$B$. Moreover, we fix any total order on the 
vertices in~$B_d$.
\end{definition}

\begin{remark}
\label{remark:quotient_nobidistance}
 Notice that in \Cref{definition:quotient_bigraph} we did not use any 
of the properties of bidistances. This means that the definition of~$B_d$ makes 
sense also for bigraphs~$B$ and pairs of functions $d_V \colon P 
\longrightarrow \Q$ and $d_W \colon Q \longrightarrow \Q$. This is 
important and useful in \Cref{formulas}.
\end{remark}

\begin{lemma}
\label{lemma:pseudo_laman_quotient}
If $B=(G,H)$ is a pseudo-Laman bigraph without self-loops and $d$ is a 
bidistance on~$B$, then the quotient graph~$B_{d}$ is also pseudo-Laman.
\end{lemma}
\begin{proof}
We first prove that for any graph $G=(V,E)$ and for any subgraph 
$G' \subseteq G$ the following equation holds:
\[
 \dim(G) \, = \, \dim(G') + \dim(G \quotient G').
\]
Let $G = \dotcup_{i=1}^{k} G_i$ be the decomposition of~$G$ into connected 
components. 
Write $G'=\dotcup_{i=1}^k G'_i$, where $G'_i$ is the part of $G'$ belonging 
to~$G_i$. 
Let $V_i$ and $V'_i$ be the set of vertices of~$G_i$ and~$G'_i$, respectively. 
Now, $G_i'$ itself may be disconnected, so let $n_i$ be the number of connected 
components of~$G_i'$.
Contraction of edges of $G_i$ does not introduce new components,
thus $G_i\quotient G_i'$ consists of one connected component.
Moreover, each connected component of $G_i'$ will correspond to one vertex in 
$G_i\quotient G_i'$.
It follows that
\[
\dim(G_i)=|V_i|-1,\quad
\dim(G_i')=|V_i'|-n_i,\quad 
\dim(G_i/G_i')=(|V_i|-|V_i'|+n_i)-1, 
\]
and therefore $\dim(G_i)=\dim(G_i')+\dim(G_i/G_i')$ for all $i$. 
Now the claim follows, because $\dim\bigl(\dotcup_{i=1}^{k} 
G_i\bigr)=\sum_{i=1}^{k} 
\dim(G_i)$.
If $B = (G,H)$ is a bigraph and $d$ is a bidistance on it
then 
\[
\dim(G) = \dim(G_{d_V}) 
\quad\textrm{ and }\quad
\dim(H) = \dim(H_{d_W}).
\] 
We prove only the first equality, the second one follows analogously. 
Let $\overline{\alpha}$ be the minimum value attained by~$d_V$. 
Then $G = G_{\geq \overline{\alpha}}$ and
\[
\dim(G) \, = \, \dim(G_{\geq\overline{\alpha}}) \, = \, \dim(G_{> 
\overline{\alpha}}) + \dim(G_{\geq \overline{\alpha}} \quotient G_{> 
\overline{\alpha}}).
\]
By repeating this argument, considering one by one all values 
in~$\operatorname{im}(d_V)$ in increasing order, we prove the asserted 
equality. 
The proof is concluded by noticing that the number of 
biedges of~$B_d$ equals the number of biedges of~$B$ by construction.
\end{proof}

\begin{example}
\label{example:B_d}
 Continuing \Cref{example:Z_K}, we fix the following bidistance $d = 
(d_V, d_W)$:
\begin{align*}
 d_V(1,2) &= 0, & d_V(1,3) &= 1, & d_V(2,3) &= 0, & d_V(2,4) &= 1, & d_V(3,4) &= 
0,\\
 d_W(1,2) &= 1, & d_W(1,3) &= 0, & d_W(2,3) &= 0, & d_W(2,4) &= 0, & d_W(3,4) &= 
1. 
\end{align*}
The bidistance~$d$ is illustrated in \Cref{figure:bidistance} and
the resulting bigraph~$B_d$ is shown in \Cref{figure:quotient_bidistance}.
The scheme $Z^{B_d}_\C$ associated to~$B_d$ is defined by the following 
equations:
\begin{gather*}
\begin{aligned}
x_{13\mid24} &= 1, & x_{13\mid24}\tth y_{1\mid2}  &= \lambda_\colcA, & 
x_{2\mid4}\tth y_{12\mid34} &= \lambda_\colcD,\\
y_{12\mid34} &= 1, & x_{1\mid3}\tth y_{12\mid 34} &= \lambda_\colcB, & 
x_{13\mid24}\tth y_{3\mid 4} &= \lambda_\colcE,
\end{aligned} \\
x_{13\mid24}+x_{24\mid13} =
  x_{1\mid3}+x_{3\mid1} =
  x_{2\mid4}+x_{4\mid2} = 0, \\
y_{12\mid34}+y_{34\mid12} =
  y_{1\mid2}+y_{2\mid1} =
  y_{3\mid4}+y_{4\mid3} = 0.
\end{gather*}
\end{example}

\begin{figure}
	\begin{center}
		\begin{subfigure}[b]{0.35\textwidth}
			\centering
			\begin{tikzpicture}[scale=0.8]
				\begin{scope} 
					\lvertex (2) at (0,0) {2};
					\lvertex (3) at (2,0) {3};
					\lvertex (1) at (1,1.73205) {1};
					\lvertex (4) at (1,-1.73205) {4};

					\draw[edge1] (1) to node[lab1,left] {0} (2);
					\draw[edge2] (1) to node[lab2,right] {1} (3);
					\draw[edge3] (2) to node[lab3,above] {0} (3);
					\draw[edge4] (2) to node[lab4,left] {1} (4);
					\draw[edge5] (3) to node[lab5,right] {0} (4);
				\end{scope}
				\begin{scope}[xshift=3.8cm] 
					\lvertex (2) at (0,0) {2};
					\lvertex (3) at (2,0) {3};
					\lvertex (1) at (1,1.73205) {1};
					\lvertex (4) at (1,-1.73205) {4};

					\draw[edge1] (1) to node[lab1,left] {1} (2);
					\draw[edge2] (1) to node[lab2,right] {0} (3);
					\draw[edge3] (2) to node[lab3,above] {0} (3);
					\draw[edge4] (2) to node[lab4,left] {0} (4);
					\draw[edge5] (3) to node[lab5,right] {1} (4);
				\end{scope}
			\end{tikzpicture}
			\caption{A bigraph on which a bidistance has been fixed.}
			\label{figure:bidistance}
		\end{subfigure}
		\quad
		\begin{subfigure}[b]{0.6\textwidth}
			\begin{center}
				\begin{tabular}{cc@{\quad}|@{\quad}cc}
					\begin{tikzpicture}[scale=0.6]
						\begin{scope} 
							\lvertex (13) at (1,1.73205) {13};
							\lvertex (24) at (1,-1.73205) {24};
							
							\draw[edge1] (13) to [bend right=40] (24);
							\draw[edge3] (13) to (24);
							\draw[edge5] (13) to [bend left=40] (24);
						\end{scope}
					\end{tikzpicture}
					&
					\begin{tikzpicture}[scale=0.6]
						\begin{scope}
							\lvertex (2) at (0,0) {2};
							\lvertex (3) at (2,0) {3};
							\lvertex (1) at (1,1.73205) {1};
							\lvertex (4) at (1,-1.73205) {4};
							
							\draw[edge2] (1) to (3);
							\draw[edge4] (2) to (4);
						\end{scope}
					\end{tikzpicture}
					&
					\begin{tikzpicture}[scale=0.6]
						\begin{scope} 
								\lvertex (12) at (1,1.73205) {12};
								\lvertex (34) at (1,-1.73205) {34};
								
								\draw[edge2] (12) to [bend right=40] (34);
								\draw[edge3] (12) to (34);
								\draw[edge4] (12) to [bend left=40] (34);
						\end{scope}
					\end{tikzpicture}
					&
					\begin{tikzpicture}[scale=0.6]
						\begin{scope}
								\lvertex (2) at (0,0) {2};
								\lvertex (3) at (2,0) {3};
								\lvertex (1) at (1,1.73205) {1};
								\lvertex (4) at (1,-1.73205) {4};
								
								\draw[edge1] (1) to  (2);
								\draw[edge5] (3) to  (4);
						\end{scope}
					\end{tikzpicture}
					\\[6pt]
					$G_{\geq 0} \bigquotient G_{> 0}$&
					$G_{\geq 1} \bigquotient G_{> 1}$&
					$H_{\geq 0} \bigquotient H_{> 0}$&
					$H_{\geq 1} \bigquotient H_{> 1}$
				\end{tabular}
			\end{center}
			\caption{The bigraph $B_d$, where $B$ and $d$ are as in 
			\Cref{example:B_d}.}
			\label{figure:quotient_bidistance}
		\end{subfigure}
	\end{center}
	\caption{A bigraph $B$ with bidistance~$d$ and the corresponding bigraph 
$B_d$}
\end{figure}
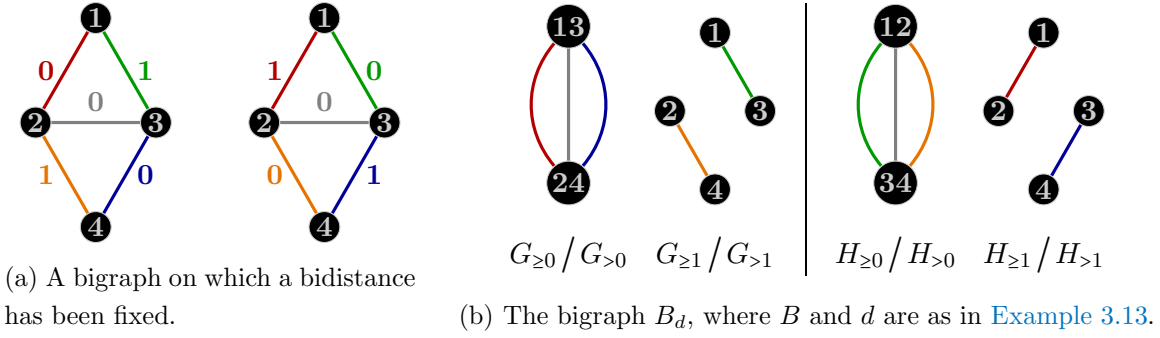

To link the points in a general fiber of~$f_{B, \K}$ that define a given 
bidistance~$d$ with the Laman number of~$B_d$, we introduce in 
\Cref{definition:family} a family 
of varieties~$\widetilde{A}_\C^d$, parametrized by a parameter~$\sigma$. This 
family has the property that a general element is isomorphic to~$Z^{B}_{\K}$, 
while a special element is isomorphic to~$Z^{B_d}_{\C}$.
To prove this, we establish in 
\Cref{lemma:induction,lemma:specialization}
the following two bijections:
\begin{gather*}
 \left\{
 \begin{array}{c}
   \text{points in } Z^{B}_{\K} \text{ that determine} \\
   \text{the bidistance } d
 \end{array}
 \right\} \underset{\text{\Cref{lemma:specialization}}}{\longleftrightarrow}
 \left\{
 \begin{array}{c}
   \text{points in } A^d_{\C}
 \end{array}
 \right\} \underset{\text{\Cref{lemma:induction}}}{\longleftrightarrow}
 \left\{
 \begin{array}{c}
   \text{points in } Z^{B_d}_{\C}
 \end{array}
 \right\} 
\end{gather*}

\begin{remark}
\label{remark:integer_values}
  Notice that, for a fixed bigraph~$B$ with biedges~$\edg$ and a fixed choice of 
a vector $\wt \in 
\Q^{\edg \setminus \{ \bar{e} \}}$ and of a bidistance $(d_V, d_W)$, we can 
suppose that all entries of the vector~$\wt$ and all values~$d_V(u,v)$ 
and~$d_W(t,w)$ are integers. Indeed, consider the subgroup of~$\Q$ 
that is generated by the rational numbers in
$\{ \wt(e) \}_{e \in \edg \setminus \{\bar{e} \}} \cup \operatorname{im}(d_V) 
\cup \operatorname{im}(d_W)$:
such a group is of 
the form $\frac{m}{n} \tth \Z$, and so we can apply the 
automorphism of~$\K$ that sends~$s$ to~$s^{n/m}$.
\end{remark}

\begin{definition}
\label{definition:family}
 Let $B$ be a bigraph with biedges~$\edg$ and use \Cref{assumption:general}. 
Given a bidistance~$d$ on~$B$, we can suppose by 
\Cref{remark:integer_values} that $\wt$, $d_V$ and $d_W$ 
take integer values. We define the scheme $\widetilde{A}^{d}_{\C}$ in $\C^{P} 
\times \C^{Q} 
\times \C$, with coordinates $( \tilde{x}_{uv} )_{(u,v)\in P}$,
$( \tilde{y}_{tw} )_{(t,w)\in Q}$ and~$\sigma$:
\[
 \widetilde{A}^{d}_{\C} \, := \, \mathrm{Spec} \!
 \left(\!
  \begin{array}{ll}
   \tilde{x}_{\bar{u} \bar{v}} = \tilde{y}_{\bar{t} \bar{w}} = 1
    & \bar{u}\prec \bar{v},\ \bar{t}\prec \bar{w} \\[2pt]
   \tilde{x}_{uv} \tth \tilde{y}_{tw} = \lambda_e
    &  \text{for all } e \in \edg \setminus \{ 
\bar{e} \},\ u\prec v,\ t\prec w \\[2pt]
   \sum_{\mscr{C}} \, \tilde{x}_{uv} \tth \sigma^{ d_V(u,v) - m(\mscr{C}) } = 
0 & \text{for all cycles } \mscr{C} \text{ in } G \! \\[2pt]
   \sum_{\mscr{D}} \, \tilde{y}_{tw}\tth \sigma^{d_W(t,w) - m(\mscr{D}) } = 
0 & \text{for all cycles } \mscr{D} \text{ in } H \!
  \end{array}
 \!\right), 
\]
where $m(\mscr{C})$ denotes the minimum value attained by the function~$d_V$ on 
the cycle~$\mscr{C}$, and similarly for~$d_W$. Since the differences
$d_V(u,v) - m(\mscr{C})$ and $d_W(t,w) - m(\mscr{D})$ are non-negative 
integers, we see that all equations are indeed polynomial in $\tilde{x}$, 
$\tilde{y}$ and~$\sigma$. Moreover, we define
\[
 A^{d}_{\C} \, := \, \widetilde{A}^{d}_{\C} \cap \{ \sigma = 0 \},
\]
where $\{ \sigma = 0 \}$ denotes the hyperplane defined by the equation $\sigma 
= 0$.
\end{definition}

\begin{example}
\label{example:A_C}
Continuing \Cref{example:B_d}, the scheme $\widetilde{A}^{d}_{\C}$ is
defined by the equations:
\begin{gather*}
\tilde{x}_{23} = \tilde{y}_{23}= 1,  \\
\begin{aligned}
\tilde{x}_{12}\tth \tilde{y}_{12} &= \lambda_\colcA,& 
\tilde{x}_{12}+\tilde{x}_{21}=
\tilde{x}_{13}+\tilde{x}_{31}=
\tilde{x}_{23}+\tilde{x}_{32}=
\tilde{x}_{24}+\tilde{x}_{42}=
\tilde{x}_{34}+\tilde{x}_{43}=0, 
\\
\tilde{x}_{13}\tth \tilde{y}_{13} &= \lambda_\colcB,& 
\tilde{y}_{12}+\tilde{y}_{21}=
\tilde{y}_{13}+\tilde{y}_{31}=
\tilde{y}_{23}+\tilde{y}_{32}=
\tilde{y}_{24}+\tilde{y}_{42}=
\tilde{y}_{34}+\tilde{y}_{43}=0, 
\\
\tilde{x}_{24}\tth \tilde{y}_{24} &= \lambda_\colcD,& 
\tilde{x}_{12}+\tilde{x}_{23}+\tilde{x}_{31}\sigma=
\tilde{y}_{12}\sigma+\tilde{y}_{23}+\tilde{y}_{31}=0, \\
\tilde{x}_{34}\tth \tilde{y}_{34} &= \lambda_\colcE,& 
\tilde{x}_{24}\sigma+\tilde{x}_{43}+\tilde{x}_{32}=
\tilde{y}_{24}+\tilde{y}_{43}\sigma+\tilde{y}_{32}=0. 
\end{aligned}
\end{gather*}
\end{example}

In \Cref{lemma:induction} we use a special set of generators for the 
ideals defining~$A^d_{\C}$ and~$Z^{B_d}_{\C}$. To describe this set of 
generators we need the concept of spanning forest for a bigraph.

\begin{definition}
\label{definition:spanning_forest}
 Let $B = (G,H)$ be a bigraph. A \emph{spanning forest}~$\mcal{F}$ for~$B$ is a 
pair $(\mcal{F}_G, \mcal{F}_H)$ of spanning forests for~$G$ and~$H$ 
respectively.
A spanning forest for a graph is a tuple of spanning trees, one for each 
connected component of the graph.
\end{definition}

An auxiliary result describing how to obtain a 
special system of generators for the ideal of~$Z^{B}_{\C}$, once a spanning 
forest for~$B$ is fixed is given in the next lemma.

\begin{lemma}
\label{lemma:equations_Z}
 Let $B$ be a bigraph. Use \Cref{assumption:general} and 
define~$Z^{B}_{\C}$ according to \Cref{definition:affine_sols}. Then 
$Z^{B}_{\C}$ is a complete intersection, and every choice of a spanning forest 
for~$B$ determines a set of $\operatorname{codim} \bigl( Z^{B}_{\C} \bigr)$ 
generators for the ideal of~$Z^{B}_{\C}$.
\end{lemma}
\begin{proof}
 Notice that the dimension of the ambient affine space of~$Z^{B}_{\C}$ is 
$|P|+|Q|$, where $P$ and $Q$ are as in \Cref{definition:pairs}. 
Moreover $Z^{B}_{\C}$ is zero-dimensional, since $\Lam(B)$ is defined.
We are going to exhibit a system consisting
of $|P|+|Q|$ equations defining~$Z^{B}_{\C}$.

Let $\mcal{F} = (\mcal{F}_G, \mcal{F}_H)$ be a spanning forest for the 
bigraph~$B=(G,H)$ with biedges~$\edg$. For every $(u,v) \in P$, where $u$
and~$v$ are not connected by an edge of~$\mcal{F}_G$, we consider the 
equation
\[
 x_{uv} - \sum_{i=0}^{n-1} x_{u_{i} u_{i+1}} \, = \, 0,
\]
where $(u_0=u,\dots,u_n=v)$ is the unique path in~$\mcal{F}_G$ from~$u$ to~$v$. 
Similarly, we construct equations for each $(t,w) \in Q$ for which $t$ and~$w$ 
are not connected by an edge of~$\mcal{F}_H$. We claim that these 
equations generate the same ideal as the equations coming from all cycles 
in~$G$ 
and in~$H$. It is enough to show this for every connected component of~$G$, so 
we can suppose that $G$ is connected. We show the claim by induction on the 
number of edges of~$G$: when this number is minimal, the graph~$G$ is a tree 
and so there is nothing to prove since there are no cycles. Suppose now that 
the statement holds for~$G$, and add an edge to~$G$ obtaining~$G'$; suppose that
this edge connects the vertices~$u'$ and $v'$. Consider an equation 
$\sum_{\mscr{C}} \, x_{uv}$ coming from a cycle~$\mscr{C}$ in $G'$: if it does 
not involve the edge $\{ u', v'\}$, then by induction hypothesis it is a linear 
combination of the equations coming from the spanning tree. Otherwise, the 
cycle~$\mscr{C}$ is of the form $(u = u_0, \dotsc, u_i, u', v', u_{i+1}, 
\dotsc, u_n = u)$. If we add to the equation $\sum_{\mscr{C}} \, x_{uv}$ the 
equation $x_{u'v'} - \sum_{j=0}^{m-1} x_{v_{i} v_{j+1}}$, we obtain the 
equation 
induced by the cycle $(u = u_0, \dotsc, u_i, v_0, \dotsc, v_m, u_{i+1}, \dotsc, 
u_n = u)$, which is completely contained in~$G$. So by induction hypothesis the 
sum is a linear combination of the equations coming from the spanning tree; 
this concludes the proof of the claim.

The number of equations coming from edges not in the spanning forest is
\begin{align*}
 |P| - |\{ \text{edges of } \mcal{F}_G \}| + |Q| - |\{ 
\text{edges of 
} \mcal{F}_G \}| 
 &= |P| - \dim(G) + |Q| - \dim(H).
\end{align*}
The above equations, together with
\[
  \begin{cases}
   x_{\bar{u} \bar{v}} = y_{\bar{t} \bar{w}} = 1, & \bar{u}\prec\bar{v},\ 
\bar{t}\prec\bar{w}, \\
   x_{uv} \tth y_{tw} = \lambda_e, & \text{for all } e \in 
\edg \setminus \{ \bar{e} \}, \, u\prec v, \  t\prec w,
  \end{cases}
\]
define~$Z^{B}_{\C}$. Therefore, the total number of equations is
$
 |\edg| + 1 + \bigl( |P| - \dim(G) \bigr) + \bigl( |Q| - \dim(H) \bigr),
$
which equals $|P| + |Q|$ since $B$ is pseudo-Laman and has no self-loops by
\Cref{assumption:general}. In particular, $Z^{B}_{\C}$ is a complete 
intersection.
\end{proof}

\begin{remark}
\label{remark:spanning_forest}
  To proceed, we need spanning forests with an additional property:
  For a bigraph $B = (G,H)$,
  we consider spanning forests $\mcal{F}_G$ and $\mcal{F}_H$ for $G$ and~$H$, 
respectively, such that for
	any edge in~$G$ with vertices $u,v$, the value $d_V(u,v)$ is equal to the 
	minimum of the values of $d_V$ in the unique path in~$\mcal{F}_G$ 
connecting~$u$ and~$v$, and similarly for~$\mcal{F}_H$.

	The construction of such forests can 
	be achieved by iteratively removing non-bridges (see 
\Cref{definition:bridge}) with endpoints~$u$ 
and~$v$ such that $d_V(u,v)$ is minimal within the non-bridges of the graph in 
the current iteration.
	This construction can be proven to be correct using the loop invariant 
$\delta \colon (V\times V) \setminus \Delta \longrightarrow \Q$, where $\Delta 
= \{ (v,v) : v \in V\}$ and $\delta(u,v)$ is defined as follows. We 
consider all paths 
$v_0=u,v_1,\ldots,v_n=v$ from~$u$ to~$v$
	and take the minimum of the values $\bigl\{ d_V(v_i,v_{i+1}) \, : \, i \in 
\{0, \dotsc, n\} \bigr\}$ for each of them.
	Then $\delta(v,u)$ is the maximum of all these values.
	Note that if $u$ and $v$ are connected by an edge in the graph, then 
$\delta(u,v)=d_V(u,v)$: in fact, by assumption the path $(u,v)$ connects $u$ 
and $v$, and so $d_V(u,v)$ appears as a minimum of a path from~$u$ to~$v$; 
moreover, for every path $v_0=u,v_1,\ldots,v_n=v$ from~$u$ to~$v$ the minimum 
of 
the values $\bigl\{ d_V(v_i,v_{i+1}) \, : \, i \in 
\{0, \dotsc, n-1\} \bigr\}$ cannot be bigger than $d_V(u,v)$, because this 
would 
contradict the property of bidistances that the minimum in the cycle 
$v_0=u,v_1,\ldots,v_n=v, u$ occurs twice; hence $d_V(u,v)$ is the maximum of 
these minima, and so it coincides with $\delta(u,v)$.
	The map~$\delta$ is indeed a loop invariant, since if we are about to delete 
an edge with endpoints~$u$ and~$v$, then this edge has to be a non-bridge of 
minimal~$d_V(u,v)$.
	Hence, there is a cycle containing both~$u$ and~$v$ and the endpoints of 
another edge with the same $d_V$\hbox{-}value, since the minimum 
$d_V$\hbox{-}value occurs at least twice in every cycle.
	Therefore, there is still a path from~$u$ to~$v$ with the same minimum, so
that the set of minima in the definition of~$\delta$ does not change at all.
	In a similar way one argues that all other values of $\delta$ are  not 
changed 
either.

The forests constructed in this way share a useful property, namely if we 
consider the set of edges
in~$B_d$ that correspond to~$\mcal{F}_G$ and~$\mcal{F}_H$, then such a set forms 
a 
spanning 
forest for~$B_d$.
\end{remark}

\begin{lemma}
\label{lemma:induction}
 Let $B$ be a bigraph. 
 Use \Cref{assumption:general} and fix a bidistance~$d$ on~$B$.
 Suppose that $B_d$ (\Cref{definition:quotient_bigraph}) 
satisfies~$\Lam(B_d)>0$.
 Then the scheme~$\widetilde{A}_{\C}^{d}$ (\Cref{definition:family})
 can be defined by~$|P| + |Q|$ equations. 
Furthermore, the scheme~$A_{\C}^{d}$ is isomorphic to~$Z^{B_d}_{\C}$, so in 
particular it consists of~$\Lam(B_d)$ distinct points and is defined by 
$\operatorname{codim} \bigl( A_{\C}^{d} \bigr)$ equations.
\end{lemma}
\begin{proof}
Let $B=(G,H)$ with biedges~$\edg$ as in the statement.
As in \Cref{lemma:equations_Z}, we can give a smallest set of equations for 
$\widetilde{A}^{d}_{\C}$
depending on a choice of a spanning forest. By a special choice of the spanning 
forest, namely by choosing
the forests $\mcal{F}_G$ and $\mcal{F}_H$ for $G$ and~$H$, respectively, as 
described in \Cref{remark:spanning_forest}, we may achieve that the 
equations are of the form
\[
 \tilde{x}_{uv} - \sum_{i=0}^{n-1} \tilde{x}_{u_{i} 
u_{i+1}} \sigma^{d_V(u_i, u_{i+1}) - d_V(u, v)} \, = \, 0
\]
together with 
\[
 \begin{cases}
  \tilde{x}_{\bar{u} \bar{v}} = \tilde{y}_{\bar{t} \bar{w}} = 1, & 
\bar{u}\prec\bar{v}, \  \bar{t}\prec\bar{w}, \\
  \tilde{x}_{uv} \tth \tilde{y}_{tw} = \lambda_e, & \text{for all } e \in 
\edg \setminus \{ \bar{e} \}, \, u\prec v, \  t\prec w. \\
 \end{cases}
\]
The number of these equations is hence $|P| + |Q|$. We obtain a set of 
equations 
for~$A^d_{\C}$ by setting $\sigma = 0$ in the previous ones. Note that we 
could not have obtained this kind of equations if we started from an 
arbitrary spanning forest for~$B$.

Let $P_d$ and~$Q_d$ be the sets as in 
\Cref{definition:pairs} starting from~$B_d$.
The elements of~$P_d$ are of the form $\bigl( [u]_{> \alpha}, [v]_{> 
\alpha} \bigr)$, where $[u]_{> \alpha}$ is the class of the vertex $u \in V$ 
in the set of vertices of $G_{\geq \alpha} \bigquotient G_{> \alpha}$, and 
$\alpha$ is a value in the image of~$d_V$. In the following we simply 
write~$[u]$ and~$[v]$ for such classes (and similarly for~$Q_d$).

We define two maps $\varphi \colon Z^{B_d}_{\C} \longrightarrow A^d_{\C}$ 
and $\psi \colon A^d_{\C} \longrightarrow Z^{B_d}_{\C}$ as follows. For a point 
$q = \bigl( ( x_{[u] [v]} )_{([u], [v]) \in P_d }, ( y_{[t] [w]} )_{([t], 
[w]) \in Q_d} \bigr)$ in~$Z^{B_d}_{\C}$, let~$\varphi(q)$ be the point 
whose $x_{uv}$\hbox{-}coordinate is~$x_{[u] [v]}$ and whose 
$y_{tw}$\hbox{-}coordinate is 
$y_{[t] [w]}$. For a point $\tilde{q} = \bigl( ( \tilde{x}_{uv} )_{(u,v) \in 
P}, ( \tilde{y}_{tw} )_{(t,w) \in Q} \bigr)$ in~$A^{d}_{\C}$, define 
$\psi(\tilde{q})$ to be the point whose $x_{[u] [v]}$\hbox{-}coordinate equals 
$\tilde{x}_{uv}$ and whose $y_{[t] [w]}$\hbox{-}coordinate 
equals~$\tilde{y}_{tw}$. We have to show that $\varphi$ and $\psi$ are 
well-defined. It is then a direct consequence of the definitions that they are 
isomorphisms.

To show that $\varphi$ 
is well-defined, we need to prove that $\varphi(q) \in A^{d}_{\C}$.
Notice that the coordinates of~$\varphi(q)$ satisfy the equations determined 
by the biedges of~$B$ because the coordinates of~$q$ do so by construction. 
Consider now an equation of~$A^d_{\C}$ obtained by setting $\sigma = 0$ in an 
equation of~$\widetilde{A}^{d}_{\C}$ determined by a cycle~$\mscr{C}$ in~$G$ 
(analogous considerations can be done for cycles in~$H$). Let $\alpha$ be the 
minimum value attained by~$d_V$ along the cycle~$\mscr{C}$.
Such an equation is of the form $\sum_{\mscr{C}_{\alpha}} \tilde{x}_{uv} = 0$, 
where the subscript in~$\mscr{C}_{\alpha}$ indicates that the sum is 
taken over the pairs~$(u,v)$ in~$P$ appearing in the cycle~$\mscr{C}$ and 
satisfying $d_V(u,v) = \alpha$. On the other hand, such a cycle determines
a cycle in~$G_{\geq \alpha} \bigquotient G_{> \alpha}$, which defines an
equation of the same form, namely $\sum_{\mscr{C}_{\alpha}} x_{[u][v]} = 0$, 
satisfied by the coordinates of~$q$. Hence $\varphi(q) \in A^{d}_{\C}$.

To show that $\psi$ is well-defined we need to first prove that if $[u]=[u']$ 
and $[v] = [v']$ for two pairs $(u,v), (u',v') \in P$ such that $d_V(u,v) = 
d_V(u',v')$, then the coordinates of the point $\tilde{q}$ satisfy 
$\tilde{x}_{uv} = \tilde{x}_{u' v'}$. This is true since by hypothesis there is 
a cycle in~$G$ involving two edges between~$u$ and~$v$ and~$u'$ and~$v'$, 
respectively, such that every other edge in the cycle has endpoints whose 
$d_V$\hbox{-}value is strictly greater than~$d_V(u,v)$. The definition
of~$\widetilde{A}^d_{\C}$ implies that such an equation holds for points 
in~$A^d_{\C}$. Secondly, we should prove that $\psi(\tilde{q}) \in 
Z^{B_d}_{\C}$, and here we argue as in the previous paragraph.
\end{proof}

The following result can be considered as a particular instance (in the 
zero-dimensional case) of the so-called \emph{tropical lifting lemma}, see 
\cite[Lemma 4.15]{Katz2009} and \cite{Payne2009} for a result over more general 
fields. We report here the proof for self-containedness and since it does not 
require results from tropical geometry, being essentially a consequence of the 
implicit function theorem for power series.

\begin{lemma}
\label{lemma:specialization}
 With the notation as in \Cref{lemma:induction}, so in particular 
a bidistance~$d$ is fixed, there is a bijection between~$A^{d}_{\C}$ and the 
set of points in~$Z^{B}_{\K}$ that determine the bidistance~$d$.
\end{lemma}
\begin{proof}
 We know from \Cref{remark:passing_to_affine_puiseux} and from 
\Cref{lemma:induction} that both $A^{d}_{\C}$ and $Z^{B}_{\K}$ consist of 
finitely many points. Let $q \in Z^{B}_{\K}$ be a point 
determining the bidistance~$d$: this means that $q = \bigl( ( x_{uv} 
)_{(u,v)\in P}, ( y_{tw} )_{(t,w)\in Q} \bigr)$ with $\nu(x_{uv}) = 
d_V(u,v)$ and $\nu(y_{tw}) = d_W(t,w)$. We can write $x_{uv} = 
\hat{x}_{uv} \tth s^{d_V(u,v)}$ and $y_{tw} = \hat{y}_{tw} \tth s^{d_W(t,w)}$ 
where the elements~$\hat{x}_{uv}$ and~$\hat{y}_{tw}$ have zero valuation. 
Therefore $\hat{q} = \bigl( ( \hat{x}_{uv} )_{(u,v)\in P}, ( \hat{y}_{tw} 
)_{(t,w)\in Q} \bigr)$ is a point of~$s^{d} \cdot Z^{B}_{\K}$, which is the
scheme in~$\K^P \times \K^Q$ defined by the equations
\[
  \begin{cases}
   \hat{x}_{\bar{u} \bar{v}} = \hat{y}_{\bar{t} \bar{w}} = 1, &
		\bar{u}\prec \bar{v},\ \bar{t}\prec \bar{w}, \\
   \hat{x}_{uv} \tth \hat{y}_{tw} = \lambda_e, &
		\text{for all } e \in \edg \setminus \{ \bar{e} \},\
		u\prec v, \ t\prec w, \\
   \sum_{\mscr{C}} \, \hat{x}_{uv} \tth s^{d_V(u,v) - m(\mscr{C})} = 0, &
		\text{for all cycles } \mscr{C} \text{ in } G, \\
   \sum_{\mscr{D}} \, \hat{y}_{tw} \tth s^{d_W(t,w) - m(\mscr{D})} = 0, & 
		\text{for all cycles } \mscr{D} \text{ in } H,
  \end{cases}
\]
where the notation is as in \Cref{definition:family}. Since all 
coordinates of~$\hat{q}$ have valuation equal to zero, we can 
define $\tilde{x}_{uv} := \hat{x}_{uv} \mod (s)$, obtaining $\tilde{x}_{uv} 
\in \C$, and similarly for $\tilde{y}_{tw}$.
It follows that the point 
$\tilde{q} = \bigl( ( \tilde{x}_{uv} )_{(u,v)\in P}, ( \tilde{y}_{tw} 
)_{(t,w)\in Q} \bigr)$ satisfies the equations of~$A_\C^{d}$. In this way we 
obtain a map from the set of points in~$Z^{B}_{\K}$ that determine the 
bidistance~$d$ to~$A^{d}_{\C}$.

Let now $\tilde{q}$ be a point of~$A^{d}_{\C}$. From 
\Cref{lemma:induction} we know that $A^{d}_{\C}$ is a complete 
intersection and that it is defined by $\operatorname{codim} \bigl( A^{d}_{\C} 
\bigr)$ equations $g_i = 0$ of the form
\[
 g_i \bigl( ( x_{uv} )_{(u,v) \in P}, ( y_{tw} )_{(t,w) \in Q} \bigr) \, = \, 
\tilde{g}_i \bigl( ( x_{uv} )_{(u,v) \in P}, ( y_{tw} )_{(t,w) \in Q}, 0 \bigr),
\]
where the equations $\tilde{g}_i = 0$ define~$\widetilde{A}^{d}_{\C}$. Since 
$A^{d}_{\C}$ is smooth by \Cref{lemma:induction}, we know that the Jacobian 
determinant
\[
 \det \left. \left( \left\{ \frac{\partial g_i}{\partial \tilde{x}_{uv}} 
\right\}, \left\{ \frac{\partial g_i}{\partial \tilde{y}_{tw}} \right\} 
\right) \right|_{\tilde{q}} 
\]
evaluated at~$\tilde{q}$ is non-zero. By the implicit function theorem for
formal power series (see~\cite[A.IV.37, Corollary]{Bourbaki2003})
applied to the system of equations $\tilde{g}_i = 0$,
there exists a unique 
point $\hat{q} \in \C[\![\sigma]\!]^P \times \C[\![\sigma]\!]^Q$ such that 
$\tilde{g}_i(\hat{q}, \sigma) = 0$ and the constant terms of the 
coordinates of~$\hat{q}$ equal the coordinates of~$\tilde{q}$. The 
point~$\hat{q}$ determines in turn a point in~$s^{d} \cdot Z^{B}_{\K}$ whose 
coordinates have valuation equal to zero, and therefore a point in~$Z^{B}_{\K}$ 
whose coordinates have valuation prescribed by~$d$. We get a map 
from~$A^{d}_{\C}$ to the set of points in~$Z^{B}_{\K}$ determining the 
bidistance~$d$.

Suppose now that there were two points~$q$ and~$q'$ in~$Z^{B}_{\K}$ determining 
the bidistance~$d$ and specializing to the same point in~$A^{d}_{\C}$. After 
applying a suitable automorphism of~$\K$ of the form $s \mapsto s^{m/n}$, we 
can 
suppose that both the points in~$s^{d} \cdot Z^{B}_{\K}$ corresponding to~$q$ 
and~$q'$ were given by power series in~$s$. This would contradict the 
uniqueness 
of the power series solution provided by the implicit function theorem.
Therefore, the two maps we have just specified provide the desired bijection.
\end{proof}

\begin{remark}
\label{remark:zero_laman}
 If $B$ is a pseudo-Laman bigraph with $\Lam(B) = 0$ and $d$ is a bidistance 
on~$B$, then the definition of~$Z^{B}_{\K}$ makes still sense, as well as the 
definitions of~$\widetilde{A}^{d}_{\C}$, $A^{d}_{\C}$ and $Z^{B_{d}}_{\C}$. In 
this case, the scheme $Z^{B}_{\K}$ is nothing but the empty set, and the proof 
of \Cref{lemma:induction} shows that the schemes 
$A^{d}_{\C}$ and $Z^{B_d}_{\C}$ are isomorphic. To conclude that $Z^{B_d}_{\C}$ 
is also the empty set we argue as in \Cref{lemma:specialization}: 
if $Z^{B_d}_{\C}$ were not empty, then one could construct a 
point in~$Z^{B}_{\K}$, contradicting the hypothesis.
\end{remark}

\begin{theorem}
\label{theorem:algebraic_counting}
 Let $B$ be a pseudo-Laman bigraph with biedges~$\edg$ without self-loops.
Fix a biedge $\bar{e} \in \edg$, and 
fix $\wt \in \Q^{\edg \setminus \{ \bar{e} \}}$. Then we have
\[
 \Lam(B) \, = \, \sum_{d} \Lam(B_d),
\]
where $d$ runs over all bidistances on~$B$ compatible with~$\wt$.
\end{theorem}
\begin{proof}
 When $\Lam(B) > 0$, the statement follows directly by combining 
\Cref{lemma:induction} and \Cref{lemma:specialization}.
The case $\Lam(B) = 0$ is covered by \Cref{remark:zero_laman}.
\end{proof}

\section{A formula for the Laman number}
\label{formulas}

In this section we develop a 
formula for the Laman number of a bigraph from 
\Cref{theorem:algebraic_counting}.
We fix a very special weight vector, namely the vector 
$(-1, \dotsc, -1)$: with this choice it is easy to determine which bidistances 
are 
compatible with~$\wt$ (\Cref{lemma:binary_distances}); the bigraphs~$B_d$ 
that one obtains are complicated, but it is possible to use this approach 
recursively (\Cref{theorem:laman_number}), translating any situation to 
a limited number of simple base cases 
(\Cref{proposition:base_cases}).
First we show in \Cref{lemma:integer_distances,lemma:binary_distances} that the 
bidistances that are compatible with 
$(-1, \dotsc, -1)$ can take only values in~$\{0,-1\}$.

\begin{lemma}
\label{lemma:integer_distances}
 Let $B$ be a pseudo-Laman bigraph with biedges~$\edg$. Suppose that ${\Lam(B) 
> 
0}$. 
Pick $\bar{e} \in \edg$ and fix $\wt \in \Z^{\edg \setminus \{ \bar{e} \} }$. 
Let $d = (d_V, d_W)$ be a bidistance for~$B$ and suppose that $\Lam(B_d) \in 
\N\setminus\{0\}$. Then both~$d_V$ and~$d_W$ take values in~$\Z$.
\end{lemma}
\begin{proof}
 Suppose by contradiction that the images of~$d_V$ and~$d_W$ 
are not contained in~$\Z$. We are going to construct an infinite family~$\{ 
d^{\kappa} \, : \, \kappa \in (0,1] \cap \Q \}$ of different bidistances 
for~$B$ that satisfies $B_{d^{\kappa}} = B_{d}$ for every $\kappa \in (0,1] 
\cap \Q$. \Cref{lemma:induction} together with 
\Cref{lemma:specialization} imply then that $Z^{B}_{\K}$ consists of 
infinitely many points, contradicting the hypothesis,
because every quotient $B_{d^\kappa}$ contributes nontrivially to $\Lam(B)$.
For every $\kappa \in (0,1] \cap \Q$, define
\begin{align*}
 d_V^{\kappa} & := \kappa \cdot d_V + (1 - \kappa) \cdot \left\lceil d_V%
\right\rceil, &
 d_W^{\kappa} & := \kappa \cdot d_W + (1 - \kappa) \cdot \left\lfloor d_W%
\right\rfloor,
\end{align*}
where $\left\lceil \cdot \right\rceil$ and $\left\lfloor \cdot \right\rfloor$ 
denote the ceiling and the floor functions, respectively. Since 
$\operatorname{im}(d_V) \cup \operatorname{im}(d_W) \not\subseteq \Z$, the
family
\[
 \bigl\{ d^{\kappa} = (d_V^{\kappa}, d_W^{\kappa}) \, : \, \kappa \in (0,1] 
\cap \Q \bigr\}
\]
has infinitely many elements.
We show that each~$d^{\kappa}$ is a bidistance for~$B$. Since by hypothesis 
\[
 d_V(u,v) + d_W(t,w) = \wt(e) \in \Z \quad \text{for all } e \in \edg \setminus 
\{ \bar{e} \},
\]
it follows that $\left\lceil d_V\right\rceil+\left\lfloor 
d_W\right\rfloor=d_V+d_W$. Hence, for all $\kappa\in (0,1]\cap\Q$
\[
 d_V^\kappa(u,v) + d_W^\kappa(t,w) = \wt(e)\quad \text{for all } e \in \edg 
\setminus\{ \bar{e} \}.
\]
By construction it 
follows that $d_V^{\kappa}(\bar{u}, \bar{v}) = d_W^{\kappa}(\bar{t}, 
\bar{w}) = 0$, where $\{\bar{u}, \bar{v}\} = \tau_G(\bar{e})$ and
$\{\bar{t}, \bar{w}\} = \tau_H(\bar{e})$.
Next, note that the two functions
\begin{align*}
 x & \longmapsto \kappa \cdot x + (1 - \kappa) \cdot \left\lceil x 
\right\rceil, &
 x & \longmapsto \kappa \cdot x + (1 - \kappa) \cdot \left\lfloor x 
\right\rfloor
\end{align*}
are strictly increasing for every $\kappa \in (0,1] \cap \Q$.
This implies that also the last property stated in 
\Cref{lemma:properties_distance}
is preserved. Hence, $d^\kappa$ is a bidistance.
Note that
\begin{align*}
  \bigl\{ (u,v)\in P \, : \, d_V(u,v)\geq \alpha \bigr\}= \bigl\{ (u,v)\in P \, 
: \, d_V^\kappa(u,v)\geq \kappa \alpha +(1-\kappa)\lceil \alpha \rceil \bigr\}
\end{align*}
and similarly for $>\alpha$ and for $d_W$.
Recall from \Cref{definition:quotient_bigraph} that $B_d$ is a disjoint union of 
graphs of the form $G_{\geq \alpha} \bigquotient G_{> \alpha}$ and $H_{\geq 
\beta} \bigquotient H_{> \beta}$.
By what we noticed, these graphs do not change when we pass from $d$ to 
$d^\kappa$ and therefore $B_{d^{\kappa}} = B_{d}$ for each~$\kappa$.
\end{proof}

\begin{lemma}
\label{lemma:binary_distances}
 Let $B$ be a pseudo-Laman bigraph with biedges~$\edg$. Suppose that $\Lam(B) > 
0$. Pick 
$\bar{e} \in \edg$ and fix $\wt = ( -1 )_{\edg \setminus \{ \bar{e} \} }$. Let 
$d 
= (d_V, d_W)$ be a bidistance for~$B$ and suppose that $\Lam(B_d) \in 
\N\setminus\{0\}$. Then both~$d_V$ and~$d_W$ take values in~$\{0,-1\}$.
\end{lemma}
\begin{proof}
 Suppose by contradiction that the claim does not hold. Then (after possibly 
swapping the roles of~$d_V$ and~$d_W$) we can suppose that $d_V(u,v) < -1$ for 
some~$u$ and~$v$ that are vertices of an edge. We construct an 
infinite 
family $\{d^{\kappa} \, : \, \kappa \in \N\}$ of bidistances for~$B$ such 
that $B_{d^{\kappa}} = B_{d}$. Then the same argument as in the proof of 
\Cref{lemma:integer_distances} gives a contradiction.

Let $\bar \alpha$ be the minimum of the values in $\operatorname{im}(d_V)$.
By \Cref{lemma:integer_distances} the value of $\bar \alpha$ is integer, so we 
have $\bar \alpha \leq 
-2$.
For any $\kappa \in \N$ define
\vspace{1ex}
\begin{align*}
 d_V^{\kappa}(u,v) & := 
  \begin{cases}
   d_V(u,v) - \kappa, & \text{if } d_V(u,v) = \bar \alpha, \\
   d_V(u,v), & \text{otherwise};
  \end{cases}\\[1ex]
 d_W^{\kappa}(t,w) & := 
   \begin{cases}
    d_W(t,w) + \kappa, & \text{if } d_W(t,w) = -1-\bar \alpha, \\
    d_W(t,w), & \text{otherwise}.
   \end{cases}
\end{align*}
\vskip 1ex
The family $\{d^{\kappa} = (d_V^{\kappa}, d_W^{\kappa}) \, : \, \kappa \in 
\N\}$ consists of infinitely many elements. From the construction it 
follows that each~$d^{\kappa}$ is a bidistance.
Furthermore,
\begin{align*}
  \bigl\{ (u,v)\in P \, : \, d_V(u,v)\geq \bar \alpha \bigr\} &= \bigl\{ 
(u,v)\in P \, : \, d_V^\kappa(u,v)\geq \bar \alpha - \kappa \bigr\}\\
  \bigl\{ (u,v)\in P \, : \, d_W(u,v)\geq -1-\bar \alpha \bigr\} &= \bigl\{ 
(u,v)\in P \, : \, d_W^\kappa(u,v)\geq -1-\bar \alpha + \kappa \bigr\}
\end{align*}
and by construction similar equalities hold for all other cases.
Here we use that $-1-\bar\alpha$ is the maximal value attained for~$d_W$.
Therefore, $B_{d^{\kappa}} = B_{d}$ by the same argument as in 
\Cref{lemma:integer_distances}.
Notice that if both~$d_V$ and~$d_W$ take values in $\{0,-1\}$, then the 
previous argument does not work. In fact in this case we have $\bar\alpha = 
-1$, and then the maps $d^\kappa$ are not bidistances anymore, since by 
construction we would have $d_W(\bar{t}, \bar{w}) = \kappa$, violating the 
prescription that $d_V(\bar{u}, \bar{v}) = d_W(\bar{t}, \bar{w}) = 0$ for the 
special biedge~$\bar{e}$.
\end{proof}

Using \Cref{proposition:splitting}, the special shape of the 
bidistances compatible with the weight vector $(-1, \dotsc, -1)$
allows to split the problem of computing the Laman number of a 
bigraph of the form~$B_d$ into the computation of the Laman numbers of two 
smaller bigraphs.

\begin{lemma}
\label{lemma:splitting_distances}
Let $B=(G,H)$ be a bigraph with biedges~$\edg$ and fix a biedge $\bar{e} \in 
\edg$.
Fix a bidistance $d = (d_V, d_W)$ such that $d_V$ and $d_W$ take values only in 
$\{-1,0\}$.
If $B$ is pseudo-Laman such that
\begin{itemize}
\item the bidistance $d$ is compatible with $\wt = ( -1 )_{\edg \setminus \{ 
\bar{e} \} }$,
\item $\bar{e}$ is neither a bridge in~$G$ nor a bridge in~$H$,
\item neither~$d_V$ nor~$d_W$ is the zero map, 
\end{itemize}
then the quotient bigraph~$B_d$ 
untangles via $\bar{e}\in\edg$
into bigraphs~$B_{d,1}$ and~$B_{d,2}$.
\end{lemma}
\begin{proof} 
Recall from \Cref{definition:quotient_bigraph} that $B$ and $B_d$ 
have the same set of biedges.
We define two sets $\edg_1, \edg_2 \subseteq \edg$ as the biedges 
in~$B_d$ corresponding to the following sets of biedges in~$B$:
\vspace{1ex}
\begin{align*}
 \left\{ e \in \edg \, :  \, 
 \begin{array}{ll}
  d_V(u,v) = 0, & \text{ where } \{u,v\} = \tau_G(e) \text{ and} \\
  d_W(t,w) = -1, & \text{ where } \{t,w\} = \tau_H(e)
 \end{array}
 \right\}, \\[1ex]
 \left\{ e \in \edg \, :  \, 
 \begin{array}{ll}
  d_V(u,v) = -1, & \text{ where } \{u,v\} = \tau_G(e) \text{ and} \\
  d_W(t,w) = 0, & \text{ where } \{t,w\} = \tau_H(e)
 \end{array}
 \right\}.
\end{align*}
\vskip 1ex
By hypothesis we have that both~$\edg_1$ and~$\edg_2$ are non-empty, and that 
$\edg = \edg_1 \cup \edg_2 \cup \{ \bar{e} \}$ is a partition, because the maps 
$d_V$ and $d_W$ take values 
only in~$\{0,-1\}$. 
In $G_{d_V}$ and $H_{d_W}$, edges with different values of $d_V$ and $d_W$, 
respectively, are in different components.
Hence the statement is proved.
\end{proof}

In order to state the final formula we introduce some notation;
we then express the bigraphs obtained from 
\Cref{lemma:splitting_distances} in terms of this new notation.

\begin{definition}
\label{definition:quotient_bigraphs}
 Let $B = \bigl( G, H \bigr)$ be a bigraph, where $G = (V, \edg)$ and $H 
= (W, \edg)$.
Given $\mcal{M} \subseteq \edg$, we define two bigraphs 
$\leftquot{\mcal{M}}{B}= \bigl( G \quotient \mcal{M}, \, H \subtract \mcal{M} 
\bigr)$
and
$\rightquot{\mcal{M}}{B}= \bigl( G \subtract \mcal{M}, \, H \quotient \mcal{M} 
\bigr)$,
with the same set of 
biedges $\edg' = 
\edg \setminus \mcal{M}$.
For both constructions we fix a total order on the vertices of the resulting 
bigraphs.
\end{definition}

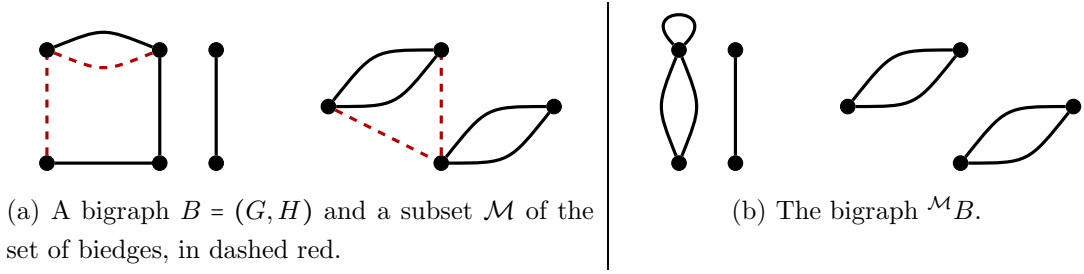
\begin{figure}
	\begin{center}
	\begin{tabular}{@{\extracolsep{\fill} } c|c}
		\begin{subfigure}[t]{0.5\textwidth}
				\begin{center}
				\begin{tikzpicture}[scale=1.5]
					\vertex (a) at (0,0) {};
					\vertex (b) at (1,0) {};
					\vertex (c) at (1,1) {};
					\vertex (d) at (0,1) {};
					\vertex (e) at (1.5,0) {};
					\vertex (f) at (1.5,1) {};

					\path
						(a) edge[edge1, dashed] (d)
						(a) edge[edge] (b)
						(b) edge[edge] (c)
						(e) edge[edge] (f)
					;
					\draw[edge1, dashed] (d) .. controls (0.5, 0.8) .. (c);
					\draw[edge] (d) .. controls (0.5, 1.2) .. (c);
				\end{tikzpicture}
				\qquad\quad
				\begin{tikzpicture}[scale=1.5]
					\vertex (a) at (0,0.5) {};
					\vertex (b) at (1,0) {};
					\vertex (c) at (1,1) {};
					\vertex (d) at (2,.5) {};

					\path
						(b) edge[edge1, dashed] (c)
						(a) edge[edge1, dashed] (b)
					;
					\draw[edge] (a) .. controls (0.4, 1) .. (c);
					\draw[edge] (a) .. controls (0.6, 0.5) .. (c);
					\draw[edge] (b) .. controls (1.4, 0.5) .. (d);
					\draw[edge] (b) .. controls (1.6, 0) .. (d);
				\end{tikzpicture}
				\end{center}
			\caption{A bigraph $B = (G, H)$ and a subset~$\mcal{M}$ of the set of 
				biedges, in dashed red.}
		\end{subfigure} &
		\begin{subfigure}[t]{0.4\textwidth}
				\begin{center}
				\begin{tikzpicture}[scale=1.5]
					\vertex (a) at (0,0) {};
					\vertex (b) at (0,1) {};
					\vertex (e) at (0.5,0) {};
					\vertex (f) at (0.5,1) {};

					\path
					(e) edge[edge] (f)
					;
					\draw[edge] (a) .. controls (-0.2, 0.5) .. (b);
					\draw[edge] (a) .. controls (0.2, 0.5) .. (b);
					\draw[edge] (b) .. controls (0.4, 1.4) and (-0.4, 1.4) .. (b);
				\end{tikzpicture}
				\qquad\quad
				\begin{tikzpicture}[scale=1.5]
					\vertex (a) at (0,0.5) {};
					\vertex (b) at (1,0) {};
					\vertex (c) at (1,1) {};
					\vertex (d) at (2,.5) {};

					\draw[edge] (a) .. controls (0.4, 1) .. (c);
					\draw[edge] (a) .. controls (0.6, 0.5) .. (c);
					\draw[edge] (b) .. controls (1.4, 0.5) .. (d);
					\draw[edge] (b) .. controls (1.6, 0) .. (d);
				\end{tikzpicture}
				\end{center}
			\caption{The bigraph ${}^{\mcal{M}} B$.}
		\end{subfigure}
	\end{tabular}
	\caption{Example of the construction in 
	\Cref{definition:quotient_bigraphs}.}
	\label{figure:quotient_bigraphs}
	\end{center}
\end{figure}

We can re-interpret \Cref{lemma:splitting_distances} in the light of 
\Cref{definition:quotient_bigraphs} by saying that if 
$d=(d_V,d_W)$ is a bidistance such that both $d_V$ and $d_W$ 
take values in $\{0,-1\}$, and $\mcal{M}$ and $\mcal{N}$ are defined as in 
\Cref{lemma:splitting_distances}, then $B_d$ untangles via $\bar e\in \edg$ 
into $\leftquot{\mcal M}{B}$ and $\rightquot{\mcal N}{B}$. This allows us 
to specialize \Cref{theorem:algebraic_counting} to a recursive formula.
By what we just said and by unraveling the notions introduced in 
\Cref{definition:quotient_bigraphs} and taking into account 
\Cref{lemma:splitting_distances} and \Cref{lemma:remove_edge} we get 
the following characterization.

\begin{proposition}
\label{proposition:characterization_bidistances}
 Let $B=(G,H)$ be a pseudo-Laman bigraph with biedges~$\edg$ without 
self-loops. Pick 
$\bar{e} \in \edg$, and fix $\wt = ( 
-1 )_{\edg \setminus \{ \bar{e} \} }$. Let $d = (d_V, d_W)$ be a bidistance 
for~$B$ such that both $d_V$ and $d_W$ take values 
in $\{0,-1\}$.
\begin{itemize}
 \item
  If $d_V$ is the zero map, then 
  $\Lam(B_d) = \Lam\left({}^{\{ \bar{e} \} } \! B\right)$.
 \item
  If $d_W$ is the zero map, then $\Lam(B_d) = \Lam\left(B^{\{ \bar{e} 
\}}\right)$.
 \item
  If neither $d_V$ nor $d_W$ is the zero map, and $\bar{e}$ is neither a bridge 
in~$G_{d_V}$ nor a bridge in~$H_{d_W}$, then 
$\Lam(B_d) = \Lam \left( \leftquot{\mcal{M}}{B} \right) 
\cdot \Lam \left( \rightquot{\mcal{N}}{B} \right)$, where $\mcal{M} \subseteq 
\edg$ is the set of biedges~$e$ such that $d_W$ is zero on~$\tau_H(e)$, and 
$\mcal{N} \subseteq \edg$ is the set of biedges~$e$ such that $d_V$ is zero 
on~$\tau_G(e)$.
\end{itemize}
\end{proposition}
\begin{proof}
If $d_V$ is the zero map, then $G_{d_V}=G$ and $H_{d_W}$ is the disjoint union 
of $H\quotient \{ \bar{e} \}$ and a single edge corresponding to 
$\{ \bar{e} \}$. If $\bar{e}$ is not a bridge in~$G$, then it is not a 
bridge in $G_{d_V} = G$ either. So, by \Cref{lemma:remove_edge} we have the 
equality $\Lam(B_d) = \Lam\left({}^{\{ \bar{e} \} } \! B\right)$. Suppose now 
that $\bar e$ is a bridge in~$G$; then by \Cref{lemma:two_bridges} we 
have $\Lam(B_d)=0$. It is therefore enough to prove that $\Lam\left({}^{\{ 
\bar{e} \} } \! B\right)=0$. We show this by proving that ${}^{\{ \bar{e} \} } 
\! B$ is not pseudo-Laman. Indeed,
\[
 \dim \bigl( H \quotient \{ \bar{e} \} \bigr) + 
 \dim \bigl( \{ \bar{e} \} \bigr) = \dim(H),
\]
as shown in the proof of \Cref{lemma:pseudo_laman_quotient}. Since $\dim \bigl( 
\{ \bar{e} \} \bigr)=1$ we have
\begin{align*}
 \dim \bigl( H \quotient \{ \bar{e} \} \bigr) + \dim \bigl( G \subtract 
\{ \bar{e} \} \bigr) &\,=\, \dim(H) - 1 + \dim(G) + 1 
 \,=\, |\edg| + 1 \, \neq \, |\edg|.
\end{align*}
Here $\dim \bigl( G \subtract \{ \bar{e} \} \bigr) = \dim(G) + 1$ because 
removing a bridge increases the dimension by~$1$. This
concludes the first case; the second is proved analogously.

Suppose now that neither $d_V$ nor $d_W$ is the zero map, and $\bar{e}$ is 
neither a bridge in~$G_{d_V}$ nor a bridge in~$H_{d_W}$.
Then by \Cref{lemma:splitting_distances} $B_d$ 
untangles, and the two bigraphs $B_{d,1}$ and $B_{d,2}$ described in that lemma 
coincide with~$\leftquot{\mcal{M}}{B}$ and~$\rightquot{\mcal{N}}{B}$.
If $B_d$ does not contain self-loops then by \Cref{proposition:splitting} we 
have
$\Lam(B_d) = \Lam \left( \leftquot{\mcal{M}}{B} \right) 
\cdot \Lam \left( \rightquot{\mcal{N}}{B} \right)$.
If $B_d$ contains a self-loop which is different from $\bar e$,
then by construction it is also a self-loop in $\leftquot{\mcal{M}}{B}$ or
$\rightquot{\mcal{N}}{B}$.
Then by \Cref{proposition:base_cases} 
$\Lam \left( \leftquot{\mcal{M}}{B} \right) 
\cdot \Lam \left( \rightquot{\mcal{N}}{B} \right)=0=\Lam B_d$.
Note that $\bar e$ might never be a loop in $B_d$.
This is because in our case $B_d$ is
\begin{equation*}
  \left( G\quotient G_{>-1} \cupdot G_{\geq 0},\, H_{\geq 0} \cupdot H\quotient 
H_{>-1} \right)
\end{equation*}
and $\bar e$ only appears in $G_{\geq 0}$ and $H_{\geq 0}$.
\end{proof}

\begin{lemma}
\label{lemma:failure_bidistance}
Let $B=(G,H)$ be a pseudo-Laman bigraph with biedges~$\edg$ without 
self-loops. Pick $\bar{e} \in \edg$, 
and fix $\wt = ( -1 )_{\edg \setminus\{ \bar{e} \} }$.
Suppose that $d = (d_V, d_W)$ is a pair of functions
$d_V \colon P \longrightarrow \{0,-1\}$ and
$d_W \colon Q \longrightarrow \{0,-1\}$ that satisfy the first three
conditions of \Cref{lemma:properties_distance}, but not the last one.
Then $\leftquot{\mcal{M}}{B}$ or $\rightquot{\mcal{N}}{B}$ has a
self-loop, where the sets $\mcal{M}$ and $\mcal{N}$ are as in 
\Cref{proposition:characterization_bidistances}.
\end{lemma}
\begin{proof}
By assumption, $d$ is not a bidistance, and it must happen that there exists a 
cycle 
in~$G$ or in~$H$ such that $d_V$ or $d_W$ attains its minimum only once. 
Let us suppose that there is a cycle~$\mscr{C}$ in~$G$ such that $d_V$ attains 
its minimum only on the pair~$(u,v)$, which is part of~$\mscr{C}$. If we 
set $\alpha = d_V(u,v)$, then we get that $G_{\geq \alpha} \bigquotient G_{> 
\alpha}$ has a self-loop. Since by definition $G_{\geq \alpha} 
\bigquotient G_{> \alpha}$ is a union of components of the graphs in either 
$\leftquot{\mcal{M}}{B}$ or $\rightquot{\mcal{N}}{B}$, the proof is completed.
\end{proof}

\Cref{proposition:base_cases} gives the two base cases for the 
computation of the Laman number of a bigraph: if the bigraph has a self-loop, 
then its Laman number is zero, and if the bigraph is constituted of two copies 
of a single edge, then its Laman number is one. They are going to be used in 
combination with the formula in \Cref{theorem:laman_number} to obtain a 
recursive algorithm. We are now able to state the formula for the computation of 
the Laman 
number of a bigraph. 

\begin{theorem}
\label{theorem:laman_number}
 Let $B=(G,H)$ be a pseudo-Laman bigraph with biedges~$\edg$ without 
self-loops. Let $\bar{e}$ be a fixed biedge, then 
\begin{equation}
\label{equation:laman_number}
  \Lam(B) = 
    \Lam \bigl( {}^{\{ \bar{e} \} } \! B \bigr) +  
    \Lam \bigl( B^{\{ \bar{e} \}} \bigr) + 
    \sum_{(\mcal{M}, \mcal{N})} 
    \Lam \bigl( \leftquot{\mcal{M}}{B} \bigr) \cdot 
    \Lam \bigl( \rightquot{\mcal{N}}{B} \bigr),
\end{equation}
where each pair $(\mcal{M}, \mcal{N}) \subseteq \edg^2$ satisfies:
\begin{itemize}
 \item 
  $\mcal{M} \cup \mcal{N} = \edg$;
 \item
  $\mcal{M} \cap \mcal{N} = \{ \bar{e} \}$;
 \item
  $|\mcal{M}| \geq 2$ and $|\mcal{N}| \geq 2$;
 \item
  both $\leftquot{\mcal{M}}{B}$ and $\rightquot{\mcal{N}}{B}$ are pseudo-Laman.
\end{itemize}
\end{theorem}
\begin{proof}
 From \Cref{theorem:algebraic_counting} we know that 
$
 \Lam(B) \, = \, \sum_{d} \Lam(B_d),
$
where $d$ runs over all bidistances on~$B$ compatible with~$\wt = ( -1 )_{\edg 
\setminus \{ \bar{e} \} }$. We distinguish two cases.

Suppose $\Lam(B)>0$. Let $d = (d_V, d_W)$ be a pair of 
functions as in \Cref{lemma:failure_bidistance}, and let $\mcal{M}, 
\mcal{N} \subseteq \edg$ be the two sets of biedges defined by~$d$. If $d$ is 
not a bidistance, then by \Cref{lemma:failure_bidistance} either 
$\leftquot{\mcal{M}}{B}$ or $\rightquot{\mcal{N}}{B}$ has a self-loop, and so 
by 
\Cref{proposition:base_cases} the contribution $\Lam \bigl( 
\leftquot{\mcal{M}}{B} \bigr) \cdot \Lam \bigl( \rightquot{\mcal{N}}{B} \bigr)$ 
is zero. If $d$ is a bidistance and $\bar{e}$ is neither a bridge in~$G_{d_V}$ 
nor a bridge in~$H_{d_W}$, then by 
\Cref{proposition:characterization_bidistances} the contribution of 
$\Lam(B_d)$ appears on the right-hand side of 
\Cref{equation:laman_number}. If instead $\bar{e}$ is a bridge 
in either~$G_{d_V}$ or in~$H_{d_W}$, then by \Cref{lemma:single_bridge} we 
conclude that $\Lam(B_d) = 0$; at the same time by 
\Cref{lemma:bridge_pseudo} either $\leftquot{\mcal{M}}{B}$ or 
$\rightquot{\mcal{N}}{B}$ is not pseudo-Laman so there is no contribution to 
the 
right-hand side of \Cref{equation:laman_number}.
A similar argument works in the case $\bar e$ is a bridge in both
$G_{d_V}$ and $H_{d_W}$ using \Cref{lemma:two_bridges}.

It remains to settle the case $\Lam(B)=0$. In this case, $\Lam(B_d)=0$ for all 
bidistances compatible with~$\wt$. We have to prove that the right hand side of 
\Cref{equation:laman_number} is zero, too. By 
\Cref{proposition:characterization_bidistances}, if $d_V$ is the zero 
map then $\Lam(B_d)=\Lam\bigl(\leftquot{\{\bar{e}\}}{B}\bigr)$. Hence, the 
first 
summand of the right-hand side of \Cref{equation:laman_number} is 
zero. For the second summand, the situation is similar. For the other summands, 
let us fix $\mcal M$ and $\mcal N$ as in the hypothesis. Define 
$d_V(\bar{u},\bar{v})=d_W(\bar{t},\bar{w})=0$, and define $d_V(u,v)=-1$ if 
there is an edge $e$ in $\mcal M$ such that $\tau_G(e)=\{u,v\}$, or 
$d_V(u,v)=0$ 
if there is no such edge; similarly for $d_W$. If $d=(d_V,d_W)$ is not a 
bidistance, then by \Cref{lemma:failure_bidistance} one of the bigraphs 
$\leftquot{\mcal{M}}{B}$ or $\rightquot{\mcal{N}}{B}$ has a self-loop and by 
\Cref{proposition:base_cases} the summand is zero. If $d=(d_V,d_W)$ 
is a bidistance, then the argument follows the same way as in the case of 
$\Lam(B)>0$. 
\end{proof}

\section{Computational results}
\label{computations}

\Cref{theorem:laman_number}, together with
\Cref{proposition:base_cases}, translates naturally
into a recursive algorithm, which has exponential complexity since it has to
loop over all subsets of~$\edg\setminus\{\bar{e}\}$. We have implemented this
algorithm~\cite{ElectronicMaterial} in the computer algebra system
Mathematica and in \verb!C++!. Despite its exponential runtime, it is a 
tremendous improvement
over the naive approach, which is to determine the number of solutions via a
Gr\"obner basis computation. For example, to compute the Laman number 880 of
the Laman graph with 10 vertices (see \Cref{figure:max_Laman_number}), our
recursive algorithm took 1.7s in Mathematica and 0.18s with \verb!C++!, while 
the
Gr\"obner basis approach took about 2353s in Mathematica and 45s
using the FGb library in Maple~\cite{FGb}.
Note also that the latter is feasible in practice only after replacing
the parameters $\lambda_e$ by random integers, which turns it into a
probabilistic algorithm. Moreover, for speed-up, we compute the Gr\"obner
basis only modulo a prime number so that the occurrence of large rational
numbers is avoided.
In contrast, our combinatorial algorithm computes the Laman number with 
certainty.
As a consistency check, we 
computed the Laman numbers of all 118,051 Laman graphs with at most 10 vertices,
using both approaches, and found that the results match perfectly.

For this purpose we generated lists of Laman graphs. In principle this is
a simple task, by applying the two Henneberg rules in all possible ways.  In
practice, it becomes demanding since one has to identify and eliminate
duplicates, which leads to the graph isomorphism problem. Using our
implementation we constructed all Laman graphs up to 12 vertices,
see \Cref{table:number_laman_graphs}.

\begin{table}
\caption{Number of Laman graphs with $n$ vertices; this sequence of numbers
is A227117 in the OEIS~\cite{Sloane}. There the sequence originally ended with
$n=8$, whose value was erroneously given as~609; we corrected and
complemented this OEIS entry accordingly.}
\begin{center}
\begin{tabular}{lccccccccccc}
\toprule
$n$ & 2 & 3 & 4 & 5 &  6 &  7 &   8 &    9 &     10 &      11 &       12 \\
\#  & 1 & 1 & 1 & 3 & 13 & 70 & 608 & 7222 & 110132 & 2039273 & 44176717 \\
\bottomrule
\end{tabular}
\end{center}
\label{table:number_laman_graphs}
\end{table}

\begin{table}
\caption{Minimal and maximal Laman number among all Laman graphs with $n$ 
vertices;
the minimum is $2^{n-2}$ and it is achieved, for example, on Laman graphs 
obtained by applying only the first Henneberg rule (see 
\Cref{theorem:laman}).}
\begin{center}
\begin{tabular}{lcccccccc}
\toprule
$n$
& 6
& 7
& 8
& 9
& 10
& 11
& 12
\\
min
& 16
& 32
& \phantom{1}64
& 128
& 256
& \phantom{2}512
& 1024

\\
max
& 24
& 56
& 136
& 344
& 880
& 2288
& 6180
\\
\bottomrule
\end{tabular}
\end{center}
\label{table:max_laman_numbers}
\end{table}

Recently, there has been large interest~\cite{Borcea2004, Emiris2009,
Emiris2013, Steffens2010, JacksonOwen2012} in the maximal Laman number that
a Laman graph with $n$ vertices can have. By applying our algorithm to all
Laman graphs with $n$ vertices, we determined the maximal Laman number for
$6\leq n\leq12$, which previously was only known for $n=6$ and $n=7$; the
results are given in \Cref{table:max_laman_numbers}. For $n=12$ this
was a quite demanding task: computing the Laman numbers of more than 44
million graphs with 12 vertices took 56 processor days using our fast \verb!C++!
implementation.

\appendix
\section{Proof of \Cref{theorem:laman}}
\label{section:prooflaman}

\begin{proof}[Proof of \Cref{theorem:laman}]
  $\ref{it:generically}\Longrightarrow \ref{it:laman}$: Assume that $G=(V,E)$ 
is 
	generically rigid. Then every 
	subgraph~$G'=(V',E')$ is generically realizable 
	(see \Cref{remark:generically_realizable}), 
	and so the map~$h_{G'}$ is dominant. 
	Therefore, the dimension of the codomain is bounded by the dimension of the 
	domain, which says $2|V'|-3 \ge |E'|$. The equality in the previous formula 
for 
	the whole graph~$G$ follows from \Cref{lemma:map_h}.

	$\ref{it:laman}\Longrightarrow \ref{it:henneberg}$: We prove the statement by 
	induction on the number of 
	vertices. The induction base with two vertices is clear. Assume that $G$ is a 
Laman graph with at 
	least~$3$ vertices. By \cite[Proposition 6.1]{Laman1970}, the graph~$G$ has a 
	vertex of degree~$2$ or~$3$. If $G$ has a vertex of degree~$2$, then the 
	subgraph~$G'$ obtained by removing this vertex and its two adjacent edges is 
a 
	Laman graph by \cite[Theorem~6.3]{Laman1970}. By induction hypothesis, $G'$ 
can 
	be constructed by Henneberg rules, and then $G$ can be constructed from 
	$G'$ by the first Henneberg rule. Assume now that $G$ has a vertex~$v$ of 
	degree~3. By \cite[Theorem~6.4]{Laman1970}, there are two vertices~$u$ 
	and~$w$ connected with~$v$ such that the graph~$G'$ obtained by removing~$v$ 
	and its three adjacent edges and then adding the edge $\{u,w\}$ is Laman. 
	By induction hypothesis, $G'$ can be constructed by Henneberg rules, and then 
	$G$ can be constructed from~$G'$ by the second Henneberg rule.

	$\ref{it:henneberg}\Longrightarrow \ref{it:generically}$:
	We prove the statement by induction on the number of Henneberg rules. The 
	induction base is the case of the one-edge graph, which is generically 
	rigid. By induction hypothesis we assume that $G=(V,E)$ is generically rigid. 
	Perform a Henneberg rule on~$G$ and let $G'$ be the result. We intend to show 
	that $G'$ is generically rigid, too.

	As far as the first Henneberg rule is concerned, we observe that for any 
	realization of~$G$, compatible with a general labeling~$\lambda$, and for any 
	labeling~$\lambda'$ extending~$\lambda$ we can always construct exactly two 
	realizations of~$G'$ that are compatible with~$\lambda'$.

	Let us now assume that $G'$ is constructed via the second Henneberg rule. 
	Call~$t$ the new vertex of~$G'$, and denote the 
	three vertices to which it is connected by~$u$, $v$ and~$w$. 
	Let $G''$ be the graph obtained by removing from~$G$ the same edge~$e$ that
	is removed in~$G'$. Without loss of generality we assume $e = \{u, v\}$. 
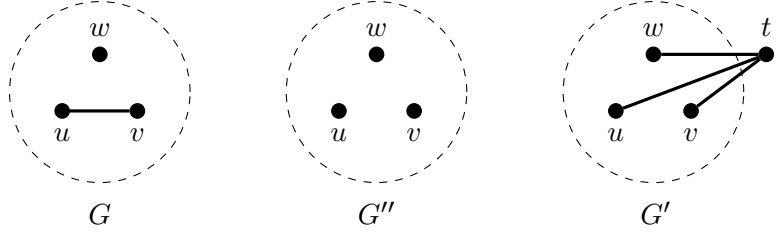
	\begin{center}
	\begin{tabular}{c@{\hspace{36pt}}c@{\hspace{36pt}}c}
	\begin{tikzpicture}
		\draw [dashed] (0,0) ellipse (1.2cm and 1.2cm);
		\vertex at (0,0.5) [label=above:$w$] {};
		\vertex (j) at (-0.5,-0.25) [label=below:$u$] {};
		\vertex (k) at (0.5,-0.25) [label=below:$v$] {};
		\path (j) edge[edge] (k);
	\end{tikzpicture}
	&
	\begin{tikzpicture}
		\draw [dashed] (0,0) ellipse (1.2cm and 1.2cm);
		\vertex at (0,0.5) [label=above:$w$] {};
		\vertex (j) at (-0.5,-0.25) [label=below:$u$] {};
		\vertex (k) at (0.5,-0.25) [label=below:$v$] {};
	\end{tikzpicture}
	&
	\begin{tikzpicture}
		\draw [dashed] (3,0) ellipse (1.2cm and 1.2cm);
		\vertex (ii) at (3,0.5) [label=above:$w$] {};
		\vertex (jj) at (2.5,-0.25) [label=below:$u$] {};
		\vertex (kk) at (3.5,-0.25) [label=below:$v$] {};
		\vertex (hh) at (4.5,0.5) [label=above:$t$] {};
		\path
		(ii) edge[edge] (hh)
		(jj) edge[edge] (hh)
		(kk) edge[edge] (hh)
		;
	\end{tikzpicture}
	\\
	$G$ & $G''$ & $G'$\hspace{12pt}
	\end{tabular}
	\end{center}
	We first show that $G'$ is generically realizable (see 
	\Cref{definition:generically_rigid}).
	Let $N \colon \C^2 
	\longrightarrow \C$ be the quadratic form corresponding to the 
	bilinear form $\left\langle \cdot, \cdot \right\rangle$.

	Fix a general labeling for~$G''$. We define the algebraic set $C \subseteq 
	\C^3$ as the set of all points $(a,b,c)$ such that there is a compatible 
	realization~$\rho$ of~$G''$ satisfying
	\[
	N \bigl( \rho(u) - \rho(v) \bigr) = a, \quad
	N \bigl( \rho(u) - \rho(w) \bigr) = c, \quad
	N \bigl( \rho(v) - \rho(w) \bigr) = b.
	\]
	For a general $a_0 \in \C$, there exist finitely many, up to equivalence, 
	points $(a_0,b,c)$ in~$C$, namely the ``lengths'' of the triangle $(u, v, w)$ 
that come from 
	the finitely many realizations of~$G$. It follows that $\dim(C) \geq 1$.

	A complex version of a classical result in distance geometry 
	(see \cite[Theorem~2.4]{Emiris2013}) states that four points $p_0$, $p_1$, 
	$p_2$, $p_3$ $\in \C^2$ fulfill
	\begin{align*}
	N(p_0-p_1) &= x, &
	N(p_0-p_2) &= y, &
	N(p_0-p_3) &= z, &
	\\
	N(p_1-p_2) &= a, &
	N(p_2-p_3) &= b, &
	N(p_1-p_3) &= c, &
	\end{align*}
	if and only if the following Cayley-Menger determinant
	\[
	F(a,b,c,x,y,z):=
	\det
	\begingroup
	\everymath{\scriptstyle}
	\renewcommand*{\arraystretch}{0.6}
	\arraycolsep=3pt
	\begin{bmatrix}
		0 & a & c & x & 1 \\
		a & 0 & b & y & 1 \\
		c & b & 0 & z & 1 \\
		x & y & z & 0 & 1 \\
		1 & 1 & 1 & 1 & 0 \\
	\end{bmatrix}
	\endgroup
	\]
	vanishes. We define
	\[
	U := \bigcup_{p=(a,b,c)\in C}S_p\,,
	\textrm{~~where~~}
	S_p := \bigl\{ (x,y,z)\in\C^3 \,\colon\, F(x,y,z,a,b,c)=0 \bigr\}
	\] 
	and
	\[
	\begin{array}{l@{~}l@{~}l@{~}l@{~}l@{~}l@{~}l@{~}l@{~}l@{~}l@{~}l@{~}l}
	e_p :=( & abc   & : & a(a-b-c) & : & b(b-a-c) & : & c(c-a-b) & : & a-b-c  
&:&\\
					& b-a-c & : & c-a-b    & : & a        & : & b        & : & c      
	&)&\in\p_\C^9. 
	\end{array}
	\]
	The point $e_p \in \p_\C^9$ with $p = (a,b,c)$ has coordinates given by the 
	coefficients of $F(a,b,c,x,y,z)$, considered as a polynomial in~$x$, $y$ 
	and~$z$. Because of this, the point $e_p$ determines~$S_p$ uniquely 
	as a surface. The function $\C^3 \setminus\{0\} \longrightarrow \p_\C^9$ 
	sending $p \mapsto e_p$ is injective, and hence the family 
$\left(S_p\right)_{p 
	\in C}$ of surfaces is not constant. It follows that the algebraic set~$U$ 
has 
	dimension~$3$, and thus a general point $(x,y,z) \in \C^3$ lies in~$U$. If we 
	extend the general labeling of~$G''$ by assigning a general triple~$(x,y,z) 
\in 
	U$ as labels to the three new edges, then we get at least one realization 
	of~$G'$. 
	It follows that $G'$ is generically realizable.

	Since $|V'|=|V|+1$ and $|E'|=|E|+2$, it follows that $2|V|=|E'|+3$. Since, as 
	we have just shown, the map~$h_{G'}$ is dominant, the graph~$G'$ is 
generically 
	rigid by \Cref{lemma:map_h}.
\end{proof}

\section{Laman graphs with maximal Laman number}

\begin{figure}[H]
\renewcommand{\vertex}{\node[svertex]}
\begin{center}
\begin{tabular}{c@{\hspace{16pt}}c@{\hspace{16pt}}c@{\hspace{16pt}}c}
 \begin{tikzpicture}[scale=0.06] 
 \draw[white] (-21,22) rectangle (22,-21);
 \vertex (a) at (-15.,-18.4482) {};
 \vertex (b) at (15,-18.4482) {};
 \vertex (c) at (0.,-6.78132) {};
 \vertex (d) at (-15.,18.4482) {};
 \vertex (e) at (15.,18.4482) {};
 \vertex (f) at (0.,6.78132) {};
 \draw[sedge] (a)edge(b) (a)edge(c) (a)edge(d) (b)edge(c) (b)edge(e) (c)edge(f) 
(d)edge(e) (d)edge(f) (e)edge(f);
 \end{tikzpicture}
 &
 \begin{tikzpicture}[scale=0.06] 
 \draw[white] (-21,22) rectangle (22,-21);
 \vertex (a) at (0.00,-17.00) {};
 \vertex (b) at (20.00,0) {};
 \vertex (c) at (7.50,0) {};
 \vertex (d) at (-20.00,0) {};
 \vertex (e) at (-7.50,0) {};
 \vertex (f) at (0.00,19.00) {};
 \vertex (g) at (0.00,7.24) {};
 \draw[sedge] (a)edge(b) (a)edge(c) (a)edge(d) (a)edge(e) (b)edge(c)
   (b)edge(f) (c)edge(g) (d)edge(e) (d)edge(f) (e)edge(g) (f)edge(g);
 \end{tikzpicture}
 &
 \begin{tikzpicture}[scale=0.06] 
 \draw[white] (-21,22) rectangle (22,-21);
 \vertex (a) at (-0.00,-17.00) {};
 \vertex (b) at (-0.00,7.24) {};
 \vertex (c) at (7.50,0.00) {};
 \vertex (d) at (-0.00,-7.24) {};
 \vertex (e) at (20.00,0.00) {};
 \vertex (f) at (-20.00,0.00) {};
 \vertex (g) at (-7.50,0.00) {};
 \vertex (h) at (-0.00,19.00) {};
 \draw[sedge] (a)edge(c) (a)edge(d) (a)edge(e) (a)edge(f) (b)edge(c)
   (b)edge(d) (b)edge(g) (b)edge(h) (c)edge(e) (d)edge(g) (e)edge(h)
   (f)edge(g) (f)edge(h);
 \end{tikzpicture}
 &
 \begin{tikzpicture}[scale=0.06] 
 \draw[white] (-21,22) rectangle (22,-21);
 \vertex (a) at (21.00,0.00) {};
 \vertex (b) at (-5.10,18.00) {};
 \vertex (c) at (0.17,-7.70) {};
 \vertex (d) at (7.90,-4.80) {};
 \vertex (e) at (-5.10,-18.00) {};
 \vertex (f) at (0.17,7.70) {};
 \vertex (g) at (-8.20,0.00) {};
 \vertex (h) at (7.90,4.80) {};
 \vertex (i) at (-19.00,0.00) {};
 \draw[sedge] (a)edge(b) (a)edge(d) (a)edge(e) (a)edge(h) (b)edge(f)
   (b)edge(g) (b)edge(i) (c)edge(d) (c)edge(e) (c)edge(f) (c)edge(g)
   (d)edge(h) (e)edge(i) (f)edge(h) (g)edge(i);
 \end{tikzpicture}
\end{tabular}\\[16pt]
\begin{tabular}{c@{\hspace{9pt}}c@{\hspace{9pt}}c}
 \begin{tikzpicture}[scale=0.09] 
 \draw[white] (-21,22) rectangle (22,-21);
 \vertex (a) at (-20.00,-16.00) {};
 \vertex (b) at (20.00,-16.00) {};
 \vertex (c) at (-0.00,-2.30) {};
 \vertex (d) at (-0.00,13.00) {};
 \vertex (e) at (-9.10,-0.20) {};
 \vertex (f) at (9.10,-0.20) {};
 \vertex (g) at (-6.00,-11.00) {};
 \vertex (h) at (-12.00,21.00) {};
 \vertex (i) at (6.00,-11.00) {};
 \vertex (j) at (12.00,21.00) {};
 \draw[sedge] (a)edge(b) (a)edge(e) (a)edge(g) (a)edge(h) (b)edge(f)
   (b)edge(i) (b)edge(j) (c)edge(e) (c)edge(f) (c)edge(g) (c)edge(i)
   (d)edge(e) (d)edge(f) (d)edge(h) (d)edge(j) (g)edge(i) (h)edge(j);
 \end{tikzpicture}
 &
 \begin{tikzpicture}[scale=0.09] 
 \draw[white] (-21,19) rectangle (22,-24);
 \vertex (a) at (7.20,4.80) {};
 \vertex (b) at (20.00,-0.00) {};
 \vertex (c) at (-7.20,4.80) {};
 \vertex (d) at (-20.00,-0.00) {};
 \vertex (e) at (0.00,-8.60) {};
 \vertex (f) at (7.20,-4.80) {};
 \vertex (g) at (0.00,-0.00) {};
 \vertex (h) at (0.00,-18.00) {};
 \vertex (i) at (-7.20,-4.80) {};
 \vertex (j) at (0.00,8.60) {};
 \vertex (k) at (0.00,18.00) {};
 \draw[sedge] (a)edge(b) (a)edge(f) (a)edge(g) (a)edge(j) (b)edge(f)
   (b)edge(h) (b)edge(k) (c)edge(d) (c)edge(g) (c)edge(i) (c)edge(j)
   (d)edge(h) (d)edge(i) (d)edge(k) (e)edge(f) (e)edge(g) (e)edge(h)
   (e)edge(i) (j)edge(k);
 \end{tikzpicture}
 &
 \begin{tikzpicture}[scale=0.09] 
 \draw[white] (-21,20) rectangle (22,-23);
 \vertex (a) at (20.00,-14.00) {};
 \vertex (b) at (-20.00,-14.00) {};
 \vertex (c) at (7.80,6.20) {};
 \vertex (d) at (-8.30,-3.90) {};
 \vertex (e) at (8.30,-3.90) {};
 \vertex (f) at (-7.80,6.20) {};
 \vertex (g) at (15.00,0.85) {};
 \vertex (h) at (20.00,14.00) {};
 \vertex (i) at (0.00,-2.10) {};
 \vertex (j) at (-20.00,14.00) {};
 \vertex (k) at (0.00,-7.80) {};
 \vertex (l) at (-15.00,0.85) {};
 \draw[sedge] (a)edge(b) (a)edge(g) (a)edge(h) (a)edge(i) (b)edge(j)
  (b)edge(k) (b)edge(l) (c)edge(d) (c)edge(g) (c)edge(h) (c)edge(j)
  (d)edge(i) (d)edge(k) (d)edge(l) (e)edge(f) (e)edge(g) (e)edge(i)
  (e)edge(k) (f)edge(h) (f)edge(j) (f)edge(l);
 \end{tikzpicture}
\end{tabular}
\end{center}
\caption{Laman graphs with $6\leq n\leq12$ vertices; for each~$n$ the graph with
the largest Laman number is shown.}
\label{figure:max_Laman_number}
\end{figure}
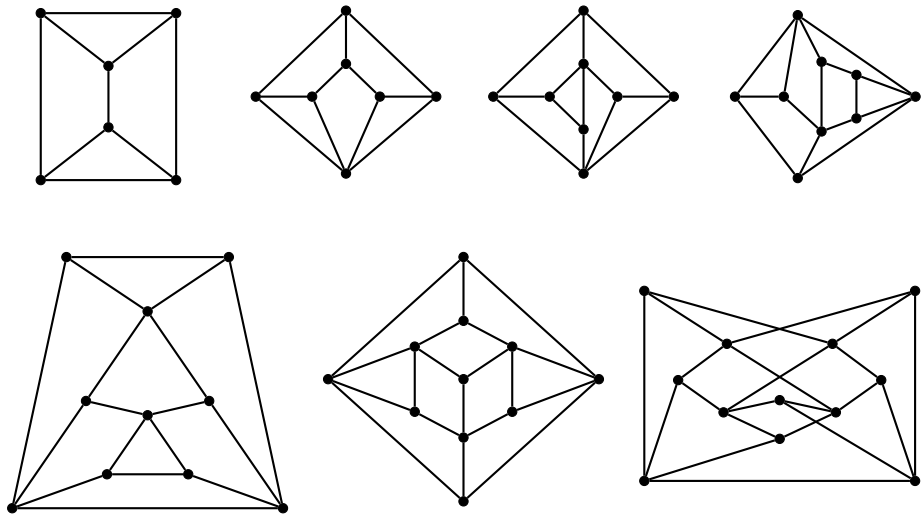

\bibliographystyle{amsalpha}

\newcommand{\etalchar}[1]{$^{#1}$}
\providecommand{\bysame}{\leavevmode\hbox to3em{\hrulefill}\thinspace}
\providecommand{\MR}{\relax\ifhmode\unskip\space\fi MR }
\providecommand{\MRhref}[2]{%
  \href{http://www.ams.org/mathscinet-getitem?mr=#1}{#2}
}
\providecommand{\href}[2]{#2}

\end{document}